\documentclass{article} 
\usepackage{amsmath,amsthm}
\usepackage{amsfonts}
\usepackage{mathtools}
\usepackage{ amssymb }
\usepackage{ stmaryrd }
\usepackage{tikz-cd}
\usepackage[shortlabels]{enumitem}
\usepackage{ upgreek }
\usepackage{yfonts}
\usepackage{dsfont}
\usepackage{url}
\usepackage{graphicx}
\usepackage{afterpage}
\usepackage{ amssymb }
\usepackage{yfonts}
\usepackage[colorinlistoftodos]{todonotes}
\usepackage{enumitem}
\setlist[enumerate]{label=(\arabic*)}
\usepackage[left=3.5cm,right=3cm,twoside]{geometry}

\usepackage{hyperref}

\def\Q{\mathbb{Q}}
\def\N{\mathbb{N}}
\def\Z{\mathbb{Z}}
\def\R{\mathbb{R}}
\def\U{\mathcal{U}}
\def\V{\mathcal{V}}

\def\sub{\subseteq}
\def\mid{\ \vert \ }

\newcommand{\s}[1]{\nsext{#1}}
\makeatletter
\DeclareRobustCommand{\nsext}[1]{%
  \binrel@{#1}
  \binrel@@{%
    {\vphantom{#1}}^*
    \kern-\scriptspace 
    \csname mkern@\detokenize{#1}\endcsname 
    {#1}
  }%
}
\newcommand{\defineextkern}[2]{%
  \@namedef{mkern@\detokenize{#1}}{\mkern#2}%
}
\makeatother

\makeatletter
\newcommand{\ostar}{\mathbin{\mathpalette\make@circled\star}}
\newcommand{\make@circled}[2]{%
  \ooalign{$\m@th#1\smallbigcirc{#1}$\cr\hidewidth$\m@th#1#2$\hidewidth\cr}%
}
\newcommand{\smallbigcirc}[1]{%
  \vcenter{\hbox{\scalebox{0.77778}{$\m@th#1\bigcirc$}}}%
}
\makeatother

\theoremstyle{definition}
\newtheorem{theorem}{Theorem}[section]

\newtheorem{coroll}[theorem]{Corollary}

\newtheorem{lemma}[theorem]{Lemma}

\newtheorem{prop}[theorem]{Proposition}

\newtheorem{defn}[theorem]{Definition}
\newtheorem*{defn*}{Definition}

\newtheorem{rmrk}[theorem]{Remark}

\newtheorem{example}[theorem]{Example}

\newtheorem{question}[theorem]{Question}

\newtheorem*{theorem*}{Theorem}

\usepackage{titling}
\usepackage[affil-it]{authblk}
\usepackage{orcidlink}
\usepackage{todonotes}
\newcommand{\email}[1]{\href{mailto:#1}{\texttt{#1}}}

\makeatletter
\newcommand{\subjclass}[2][2020]{%
  \let\@oldtitle\@title%
  \gdef\@title{\@oldtitle\footnotetext{\hspace*{-2em}#1 \emph{Mathematics subject classification.} #2}}%
}
\newcommand{\keywords}[1]{%
  \let\@@oldtitle\@title%
  \gdef\@title{\@@oldtitle\footnotetext{\hspace*{-2em}\emph{Keywords:} #1.}}%
}

\makeatother

\author[1]{Lorenzo Luperi Baglini\orcidlink {0000-0002-0559-0770}%
  \thanks{email: \email{lorenzo.luperi@unimi.it}}}

\author[2]{Alessandro Vegnuti%
  \thanks{email: \email{alessandro.vegnuti@unimi.it}}}

\affil[1]{\small Dipartimento di Matematica, Università  di Milano, Via Saldini 50, 20133 Milano, Italy}

\affil[2]{\small Dipartimento di Matematica, Università  di Milano, Via Saldini 50, 20133 Milano, Italy and Mathematisches Institut, Albert-Ludwigs-Universität Freiburg, D-79104 Freiburg,
Germany}

\title{Asymptotic Ramsey theory of Diophantine equations}



\begin{document}

\maketitle

\begin{abstract}
    We introduce the notion of asymptotic partition regularity for Diophantine equations. We show how this notion is at the core of almost all known negative results in the Ramsey theory of equations, and we use it to produce new ones, as in the case of Fermat-Catalan equations. The methods we use here are based on translating asymptotic partition regularity into the context of nonstandard extensions, via the notion of Archimedean equivalence classes of hypernaturals. 
\end{abstract}

\section*{Introduction}

With Arithmetic Ramsey Theory it is usually indicated a family of theorems asserting that certain arithmetic configurations are Partition Regular (PR from now on\footnote{PR will mean ``partition regular'' or ``partition regularity'', depending on the context.}), where a configuration is PR when given any finite partition of the positive integers $\N=\{1,2,\dots\}$, one can always find that configuration contained in a single piece of the partition. Usually, a partition is also called a \textit{coloring}, every piece \textit{color} and a pattern contained in a single color is also called \textit{monochromatic}.

Historically, the first result in the area is Schur's Theorem (\cite{schur1917kongruenz}), which states that in every finite coloring of $\N$ there are always $a,b$ such that $\{a,b,a+b\}$ is monochromatic. Later on, in 1927 Van der Waerden proved that, for every fixed $l\in \N$, you can always find $a,b$ such that $\{a, a+b,\dots, a + l b\}$ is monochromatic (\cite{van1927beweis}); and in 1933 Rado (\cite{Rado}) fully characterized which linear equations are PR. Results on the nonlinear case have been more scarce; except for a few fundamental works like \cite{bergelson1996polynomial} and \cite{sarkozy1978difference}, this line of research has been mostly active in the past fifteen years. In particular, several new results have been obtained using methods coming from logic, namely studying the problem from the point of view of nonstandard extensions, see e.g. \cite{baglini2013partition}, \cite{IteratedHyper} and \cite{di2018ramsey}.\\

In this paper, we introduce a new notion of partition regularity, called \textit{asymptotic PR}, which takes into account the mutual largeness of elements of monochromatic solutions.

\begin{defn}\label{defn:asymPR} Let\footnote{In the following, we will use the symbol $\uplus$ to denote a disjoint union.} $I_{1}\uplus\dots \uplus I_{s}=\{1,\dots,n\}$.  We say that $P\left(x_1,\dots, x_n\right)=0$ is asymptotically PR in $I_{1},\dots,I_{s}$ if for every finite partition of $\N$, for every $N\in\N$ there are monochromatic solutions $a_1,\dots, a_n$ such that:
    \begin{enumerate}
        \item for all $r\leq s$, for all $i,j\in I_{r}$ $\vert\frac{a_i}{a_j}-1\vert<\frac{1}{N}$;
        \item for all $r,t\in\{1,\dots,s\}$, for all $i\in I_{r}, j\in I_{t}$, if $r<t$ then $N a_{j}< a_{i}$.
    \end{enumerate}
    The sets $I_{i}$ will be called asymptotic classes.
    \end{defn}

Basically, $P(x_1,\dots, x_n)=0$ is asymptotically PR in $I_{1},\dots,I_{s}$ when monochromatic values in $I_{i}$ can be chosen relatively close to each other, whilst values belonging to different $I_{j}$ can be chosen arbitrarily distant from each other.

Our first simple, but relevant, result is that partition regularity and asymptotic PR are highly related.

\begin{theorem*}[A] Let $P\left(x_1,\dots, x_n\right)\in\Z[x_{1},\dots,x_{n}]$. The following facts are equivalent:
\begin{enumerate}
    \item $P\left(x_{1},\dots,x_{n}\right)=0$ is PR;
    \item $P\left(x_{1},\dots,x_{n}\right)=0$ is asymptotically PR in $I_{1},\dots,I_{s}$ for some $s\leq n, I_{1},\dots,I_{s}\subseteq\{1,\dots,n\}$.    
\end{enumerate}
\end{theorem*}

Asymptotic PR seems to be particularly well-suited to find necessary conditions for the PR of equations, especially in the nonlinear case. In fact, we will use it to provide almost immediate proofs of (at the best of our knowledge) all main known necessary conditions for the PR of Diophantine equations, in particular of

\begin{itemize}
    \item Rado's Theorem for linear equations (\cite{Rado}, Theorem 4), here Theorem \ref{2 Rado};
    \item Di Nasso-Luperi Baglini's homogeneity condition (\cite{di2018ramsey}, Theorem 3.10), here Theorem \ref{2 Pezzo omogeneo alla Rado};
    \item Barret-Lupini-Moreira's maximal Rado condition (\cite{BarrettLupiniMoreira}, Theorem 3.1), here Proposition \ref{1 Teorema sulle teste}.
\end{itemize}

We use asymptotic PR also to produce some new results, the main being a very strict necessary condition for the PR of Fermat-Catalan equations (Theorem \ref{4 th x^n - y^n}). 

\begin{theorem*}[B]Let $n\geq 2$, let $a,b\in\Z$ and let $P(z)\in \Z[z]$ be a polynomial of degree $k$. If $n\notin \{k,k-1\}$ then $ax^n+by^n = P(z)$ is not PR, if not trivially (i.e. with constant solutions).
\end{theorem*}

Nonstandard methods have been intensively used in this area of research in the past 15 years, and are usually based on the well-known notion of $u$-equivalence $\sim$ (see Definition \ref{1 def sim}). Being a non-Archimedean structure, $\s\N$ also has its own notion of Archimedean equivalence $\asymp$. Our main result characterizing asymptotic PR is that $\sim$ and $ \asymp$ interact as specified in the following Theorem, that will be our main tool in this paper.

\begin{theorem*}[C]
 Let $P\left(x_{1},\dots,x_{n}\right)\in \Z\left[x_{1},\dots,x_{n}\right]$, and let $I_{1}\uplus\dots \uplus I_{s}=\{1,\dots,n\}$. The following are equivalent:
    \begin{enumerate}
        \item $P(x_1,\dots, x_n)= 0 $ is asymptotically PR in $I_{1},\dots,I_{s}$;
        \item There exists $(\alpha_1,\dots, \alpha_n)\in \s\N^n$ such that:
        \begin{itemize}
        \item[(i)]  for all $i,j\leq n,\  \alpha_i\sim\alpha_j$;
            \item[(ii)] for all $r\leq s$, for all $i,j\in I_{r}$ $\alpha_{i}\asymp \alpha_{j}$;
            \item[(iii)] for all $1\leq r < t \leq s$, for all $i\in I_{r},j\in I_{t}$ $\alpha_{i}\gg\alpha_{j}$;
            \item[(iv)] $P\left(\alpha_{1},\dots,\alpha_{n}\right)=0$.
        \end{itemize}
    \end{enumerate}  
     
\end{theorem*}

The paper is structured as follows. Section 1 is divided in three parts: in the first, we prove some basic standard facts regarding asymptotic PR, including Theorem (A) and a complete characterization of the linear case; in the second, we recall the basic properties of $\sim$ that we need; in the third, we relate asymptotic PR with Archimedean classes of hypernaturals and prove our Theorem (C). Section 2 is devoted to find necessary conditions for the PR of equations. More precisely, in Section 2.1 we give the promised new proof of several known such conditions, whilst in Section 2.2 we prove Theorem (B) and some similar negative results. In Section 3 we relate asymptotic PR to ultrafilters: in Section 3.1 we recall the relation between $\s\N$ and $\beta\N$, in Section 3.2 we provide very compact proofs of classical results about linear equations involving ultrafilters, and in Section 3.3 we provide a characterization of the dual notion of Archimedean classes of ultrafilters. Finally, in section 4 we present some interesting open problems that arose from our work.

\paragraph{Funding}

LLB was supported by the project PRIN 2022 ``Logical methods in combinatorics'', 2022BXH4R5, Italian Ministry of University and Research (MUR).

\section{Asymptotic Partition Regularity and Archimedean classes of hypernaturals}

\subsection{Asymptotic Partition Regularity}\label{sec stand asymp}

As said in the introduction, asymptotic PR, as defined in Definition \ref{defn:asymPR}, is our main topic in this paper. We start by proving Theorem (A) from the introduction, connecting asymptotic PR and PR. 
\begin{theorem}\label{thmasympclas} Let $P\left(x_1,\dots, x_n\right)\in\Z[x_{1},\dots,x_{n}]$. The following facts are equivalent:
\begin{enumerate}
    \item $P\left(x_{1},\dots,x_{n}\right)=0$ is PR;
    \item $P\left(x_{1},\dots,x_{n}\right)=0$ is asymptotically PR in $I_{1},\dots,I_{s}$ for some $s\leq n, I_{1},\dots,I_{s}\subseteq\{1,\dots,n\}$.    
\end{enumerate}
\end{theorem}

\begin{proof} $(1)\Rightarrow (2)$ By contrast: if not, consider all the $B_{n}$ possible finite partitions of $\{1,\dots,n\}$, where $B_{n}\in\N$ is the $n$-th Bell number. Enumerate this set of partitions as $\{P_{1},\dots,P_{B_{n}}\}$. For each $i\leq B_{n}$, by our hypothesis there exist a finite partition $Q_{i}$ of $\N$ and a finite number $N_{i}$ such that for all $Q_{i}$-monochromatic solution of $P\left(x_{1},\dots,x_{n}\right)=0$ at least one of the conditions in Definition \ref{defn:asymPR} fails for $P_{i},N_{i}$. 

Let $Q$ be a common refinement of $Q_{1},\dots,Q_{B_{n}}$, i.e. any set that is $Q$-monochromatic is also $Q_{i}$-mochromatic for all $i\leq B_{n}$, and let $N:=\max\{N_{i}\mid i\leq B_{n}\}$. As $P\left(x_{1},\dots,x_{n}\right)=0$ is PR, there exists a $Q$-monochromatic solution $a_{1},\dots,a_{n}$. Let $a^{(1)}=\max\{a_{i}\mid i\leq n\}$ and let

\[I_{1}=\left\{a_{i}\mid i\leq n, \left\vert\frac{a^{(1)}}{a_i}-1\right\vert<\frac{1}{N}\right\}\]

and, inductively, assuming to have defined $I_{1},\dots, I_{j}$, let $a^{(j)}=\max\{a_{i}\mid i\in \{1,\dots,n\}\setminus\left(I_{1}\uplus\dots \uplus I_{j}\right)\}$ and let 

\[I_{j+1}=\left\{a_{i}\mid i\in \{1,\dots,n\}\setminus\left(I_{1}\uplus\dots \uplus I_{j}\right), \left\vert\frac{a^{(j)}}{a_i}-1\right\vert<\frac{1}{N}\right\}.\]

This gives a partition of $\{1,\dots,n\}$. Assume that, in our enumeration of the partitions of $\{1,\dots,n\}$, this partition is $P_{i}$.
By construction, $a_{1},\dots,a_{n}$ is a $Q$-monochromatic solution of $P(x_{1},\dots,x_{n})=0$ that satisfies conditions $(1)$ and $(2)$ of Definition \ref{defn:asymPR} for $P_{i},N$. Hence, as $Q$ refines $Q_{i}$ and $N\geq N_{i}$, $a_{1},\dots,a_{n}$ is also a $Q_{i}$-monochromatic solution of $P(x_{1},\dots,x_{n})=0$ that satisfies conditions $(1)$ and $(2)$ of Definition \ref{defn:asymPR} for $P_{i}, N_{i}$, which is a contradiction.

$(2)\Rightarrow (1)$ This is immediate.\end{proof}

\begin{example}\label{ex semplici lineari} Let us consider two simple examples.
\begin{enumerate}
    \item The simplest example of a PR equation is Schur's equation $x+y=z$. It is easily seen that this equation can only be asymptotically PR in $\{x,z\},\{y\}$ or $\{y,z\},\{x\}$: in fact, $z$ is clearly larger than $x,y$, but it cannot be $I_{1}=\{z\}$, as this would force $x,y$ to be in smaller asymptotic classes, hence $x+y$ to be smaller than $z$. Moreover, it cannot be $I_{1}=\{x,y,z\}$ as, for example, $\vert\frac{x}{y}-1\vert<\frac{1}{2}$ forces (as $z=x+y$) $\vert \frac{z}{y}-1\vert >\frac{1}{2}$, against condition (1) in Definition \ref{defn:asymPR}. As $x,y$ can clearly be exchanged in this equation, this example shows also that the asymptotic classes are not uniquely determined, in general.
    \item With similar arguments one can show that, if the Pythagorean equation $x^2 + y^2=z^2$ is PR, then it must be asymptotically PR in $\{x,z\}, \{y\}$ or $\{y,z\},\{x\}$.
    \item Let us consider now the PR equation $x+2y=z$. With similar arguments as above, it is simple to show that the asymptotic classes in this case are necessarily $I_{1}=\{x,z\}$ and $I_{2}=\{y\}$.
\end{enumerate}\end{example}

In the linear case, asymptotic PR can be fully characterized. In fact, the asymptotic conditions in Definition \ref{defn:asymPR} have two forms: conditions on variables in the same asymptotic class, and conditions on variables in different asymptotic classes.

Saying that $x_{i},x_{j}$ belongs to the same asymptotic class is equivalent to the fact that, for all $N$, we can solve monochromatically the system made of our original equation, with the addition of the inequalities

\begin{equation}
    \begin{cases}
      Nx_{i}-x_{j}>0, \\
      -x_{i}+Nx_{j}>0.\end{cases}
\end{equation}
When $x_{i},x_{j}$ belong to two different asymptotic classes, say $x_{i}$ in the largest, we have to be able, for all $N\in\N$, to solve monochromatically the original equation under the additional condition $x_{i}-x_{j}\geq N$; in the language introduced in \cite{LBARadoFunctionals}, this amounts to be able to solve monochromatically the system
\begin{equation}\label{sto troiaio}
    \begin{cases}
      P\left(x_{1},\dots,x_{n}\right)=0, \\
      x_{i}-x_{j}=\infty.\end{cases}
\end{equation}

In \cite[Lemma 2.18]{LBARadoFunctionals}, the authors proved that, when $P\left(x_{1},\dots,x_{n}\right)$ is actually a system of linear equations, the PR of system (\ref{sto troiaio}) is equivalent to the PR of the following simpler one:
\begin{equation}\label{sto troiaio}
    \begin{cases}
      P\left(x_{1},\dots,x_{n}\right)=0, \\
      x_{i}-x_{j}>0.\end{cases}
\end{equation}

When $P\left(x_{1},\dots,x_{n}\right)$ is a linear equation\footnote{A similar argument would work for systems of linear equations, but we will not consider such a case in this paper.}, a characterization of the PR regularity of systems where $P\left(x_{1},\dots,x_{n}\right)=0$ has to be solved under additional linear inequalities constraints was given by Hindman and Leader in \cite[Theorem 2]{HindmanLeaderIneq}, that we recall.

\begin{theorem}[\cite{HindmanLeaderIneq}, Theorem 2]\label{HLI}
    Let $A$ be an $m\times n$ rational matrix, and let $b^{(j)}_i$, $i = 1,..., n, j = 1,..., d$ be rationals. Then the following are equivalent: \begin{itemize}
        \item for every finite coloring of $\N$ there is a monochromatic vector $x \in \N^n$ with $Ax = 0$
        and $\sum_{i=1}^n b_i^{(j)} x_i \geq 0$ for $j\leq d$;
        \item there are positive rationals $q_1,\dots, q_d$ such that the system of equations formed from
        $Ax = 0$ by adding new variables $z_1,\dots,z_d$ and new equations $\sum_{i=1}^n b_i^{(j)} x_i - q_j z_j= 0$ for $j\leq d$ is partition regular.
    \end{itemize}
\end{theorem}

By using the above result, we can now show that linear equations can always be solved with just two asymptotic classes: the largest one corresponding to variables whose coefficients sums to zero, and the smallest one containing all the remaining ones.

\begin{theorem}\label{thm char linear asymp}
Let $c_{1},\dots,c_{n}\in\Z\setminus\{0\}$, let $k\leq n$ and assume that $\sum_{i=1}^{k} c_{i}=0$. Then the linear equation $\sum_{i=1}^{n}c_{i}x_{i}$ is asymptotically PR in $I_{1}=\{1,\dots,k\}, I_{2}=\{k+1,\dots,n\}$.\end{theorem}

\begin{proof} By the above discussion, Theorem \ref{HLI} and Rado's characterization of the PR of linear systems (see \cite{Rado}), our thesis is equivalent to prove that for all $N\in\N$, for all $2\leq i\leq k$, $k+1<j\leq n$ there exist $q_{1,i}, q_{i,1}, q_{k+1,j}, q_{j,k+1}, q_{1,k+1}\in \mathbb{Q}_{>0}$ such that the $2n\times (3n-1)$ matrix\\

\hspace{-0.5cm}\resizebox{\linewidth}{!}{
$\displaystyle\begin{pmatrix}
        c_{1} & c_{2} & c_{3} & \dots & c_{k} & c_{k+1} & c_{k+2} & \dots & c_{n} & 0 & 0 & \dots & \dots &0 & 0& \dots & \dots&0  \\
        N & -1 & 0 & \dots & 0 &  0 & 0 &  \dots & 0 & -q_{1,2} & 0 &\dots & \dots & 0& 0&\dots & \dots&0\\
        -1 & N & 0 & \dots & 0 &  0 & 0 &  \dots & 0 & 0 & -q_{2,1}  & \dots&\dots & 0 &0 &\dots &\dots& 0 \\
\vdots&\vdots&\vdots&\vdots&\vdots&\vdots&\vdots&\vdots&\vdots&\vdots&\vdots&\vdots&\vdots&\vdots&\vdots&\vdots&\vdots&\vdots\\
        N & 0 & 0 & \dots & -1 &  0 & 0 &  \dots & 0 & 0 & 0 &\dots  & -q_{1,k} &0 & 0& \dots& \dots&0\\
        -1 & 0 & 0 & \dots & N &  0 & 0 &  \dots & 0 & 0 & 0 &\dots  & 0 & -q_{k,1} & 0& \dots&\dots& 0\\
        0 & 0 & 0 & \dots & 0 &  N & -1 &  \dots & 0 & 0 & 0 &\dots  & 0 & 0& -q_{k+1,k+2} & 0& \dots& 0\\
\vdots&\vdots&\vdots&\vdots&\vdots&\vdots&\vdots&\vdots&\vdots&\vdots&\vdots&\vdots&\vdots&\vdots&\vdots&\vdots&\vdots&\vdots\\
        0 & 0 & 0 & \dots & 0 &  -1 & 0 &  \dots & N & 0 & 0  & \dots & \dots & \dots &\dots &\dots & -q_{n,k+1}& 0 \\   
        1 & 0 & 0 & \dots & 0 &  -1 & 0 &  \dots & 0 & 0 & 0  & \dots & \dots & \dots &\dots &\dots& 0& -q_{1,k+1}    \end{pmatrix}$
        }\\

satisfies the columns condition. This is attained by letting $q_{1,k+1}=1$ and $q_{a,b}=N-1$ in all the other cases. In fact, with this choice, if for $i\leq 3n-1$ we let $C_{i}$ be the $i$-th column in the matrix, and we let $D=\{1,\dots,k,n+1,\dots,n+2k-2,3n-1\}$ we see that $\sum_{i\in D} C_{i} =0$, and $\sum_{i\in [1,3n-1]\setminus D} C_{i}$ belongs to Span$\{C_{i}\mid i\in D\}$ trivially.
\end{proof}

 \begin{rmrk} Let us remark explicitly that Theorem \ref{thm char linear asymp} does not state that the asymptotic classes in the linear case must necessarily be two. In fact, given a linear equation, for any subset of its coefficients that sums to zero we have a different partition in two asymptotic classes; this is the case of $x+y=z$, as discussed in the Example \ref{ex semplici lineari}).(1). 
 
 Moreover, there are linear equations that admit more than two asymptotic classes. For example, let $n\geq 3$ and consider the PR equation \begin{equation}\label{eq n classi}x_{1}-y_{1}+z_{1}+x_{2}-y_{2}+z_{2}+\dots+x_{n}-y_{n}+z_{n}=0.\end{equation} As $x-y=z$ is PR and admits, in any coloring, monochromatic solutions with arbitrarily large entries, arguing like in Example \ref{ex semplici lineari}.(1) it is immediate to see that Equation (\ref{eq n classi}) is aymptotically PR in $\{x_{1},y_{1}\},\{z_{1}\},\dots,\{x_{n},y_{n}\},\{z_{n}\}$, hence with $2n$ asymptotic classes.

 Finally, there are cases where the asymptotic classes are, indeed, uniquely determined, for example the equation $x+2y=z$ discussed in the Example \ref{ex semplici lineari}.(3).
 \end{rmrk}

Whilst standard methods could be used to study the basic properties of asymptotic PR, we prefer to adopt techniques coming from nonstandard analysis, as these will allow to simplify most of our arguments. Before doing so, we recall some well-known facts regarding the nonstandard approach to Ramsey theory.

\subsection{Preliminaries: $u$-equivalence}\label{sec:nsa}

    We assume the reader to be familiar with the basics of nonstandard analysis, see e.g. \cite{DGL} or \cite{Goldblatt}. In what follows, we work in a nonstandard extension $\s\N$ of $\N$, which we assume to be at least $2^{\aleph_{0}}$-saturated. In such a setting, it is possible to introduce the notion of $u$-equivalence. We refer to \cite{baglini2012hyperintegers,baglini2013partition,di2015hypernatural,di2018ramsey,DGL} for extended presentations of this nonstandard approach.

\begin{defn}\label{1 def sim}
    Let $k\in\N$ and let $\alpha, \beta\in\s\N^{k}$. We say that $\alpha,\beta$ are $u$-equivalent if for all $A\sub \N^{k}$ $$\alpha \in \s A \Leftrightarrow \beta \in \s A.$$
    We write $\alpha \sim \beta$ when $\alpha,\beta$ are $u$-equivalent.
\end{defn}

\begin{rmrk} Readers with a background in model theory will easily recognise that $\alpha \sim \beta$ if and only if they have the same type in the full language\footnote{Namely the language having a relation symbol for every subset of $\N^{k}$, for all $k\in\N$.} over $\N$. 
\end{rmrk}

Being $u$-equivalent can also be rephrased as follows: given $\alpha\in\s\N^{k}$, let $\U_{\alpha}=\{A\subseteq\N^{k}\mid \alpha\in \s A\}$. Then $\U_{\alpha}$ is an ultrafilter over $\N^{k}$, $\alpha$ is called a generator of $\U_{\alpha}$ and $\alpha\sim\beta$ if and only if $\U_{\alpha}=\U_{\beta}$, i.e. if and only if they generate the same ultrafilter. Subsection \ref{sec:ultrafilters and hyperextensions} will be devoted to give more details about the relation between numbers in $\s\N$ and ultrafilters over $\N$. 

\begin{defn} We say that $\alpha$ is a generator of $\U$ if $\U=\U_{\alpha}$, and write $\alpha\models \U$.\end{defn}

The identification of ultrafilters with their nonstandard generators is the idea behind Theorem 2.2.9 in \cite{baglini2012hyperintegers}, that we recall here in a slightly rephrased form for the easiness of the reader\footnote{More general versions of this theorem for the PR of arbitrary configurations hold, see e.g. Theorem 2.2.11 in \cite{baglini2012hyperintegers}.}. 

\begin{theorem}\label{2 Teo fondamentare PR sse esistono soluzioni equiv}
    $E(x_1,\dots, x_n)=0$ is PR if and only if there exist $\alpha_1\sim\dots \sim \alpha_n\in \s\N$ such that $$E(\alpha_1,\dots, \alpha_n)=0.$$
\end{theorem}

Theorem \ref{2 Teo fondamentare PR sse esistono soluzioni equiv} allows to reduce the study of the partition regularity of equations to the study of properties of types, which is usually performed using the following well-known properties.

\begin{lemma}\label{2 proprietà generali sim}
    Let $\alpha\in \s\N^{n},\beta\in \s\N^{m}$, and let $f:\N^{n}\rightarrow \N^{m}$. 
    \begin{enumerate}
                \item if $\alpha \sim\beta$ then $f(\alpha)\sim f(\beta)$;
        \item if $\alpha \sim f(\alpha)$ then $\alpha =f(\alpha)$; in particular, if $n=m=1$ and $\alpha\sim\beta$ then $\alpha=\beta$ or $|\alpha-\beta|$ is infinite;
        \item if $f(\alpha)\sim \beta$ then there exists $\gamma\sim \alpha$ such that $\beta=f(\gamma)$;
        \item if $\alpha\sim\beta$ and $\alpha\in \N^{n}$ then $\alpha=\beta$.
    \end{enumerate}
\end{lemma}

 The interested reader can find a proof of the above properties, e.g., in \cite{di2015hypernatural}.

\subsection{Archimedean classes of hypernaturals}

In any non-Archimedean setting the following notion, that we specialise here to the $\s\N$ case, arises naturally.

\begin{defn}
    Two numbers $\alpha,\beta \in \s\N$ are in the same Archimedean class if there exist $n_1, n_2\in\N$ such that $\frac{\beta}{n_1}< \alpha < n_2\beta $. In this case, we write $\alpha \asymp \beta$, and define $Arc(\alpha)=\{\beta \in \s\N \mid \beta \asymp \alpha\}$.

    We say that $\alpha \ll \beta$ if for every $n\in \N$, $n\alpha < \beta$.
\end{defn}

Notice that without loss of generality we may suppose $n_1=n_2=n$; it is also easy to check that $\asymp$ is an equivalence relation, and that given $\alpha, \beta$ then $\alpha \asymp \beta$ or $\alpha \ll \beta$ or $\alpha \gg \beta$.

The $\asymp$-equivalence is related to integer logarithms. For the sake of clarity, let us recall here some fundamental properties of integer logarithms that we will often use without mentioning.

\begin{rmrk}\label{basic prop log}
    For $n\in \N$, $n\geq 2$, let $l_{n}:=\lfloor\log_n\rfloor$. The following properties hold:
    \begin{itemize}
        \item for every $x\in \N$, $n^{l_n(x)}\leq x<n^{l_n(x) +1}$;
        \item for every $x,y\in \N $, $l_n(xy)=l_n(x) + l_n(y) + \delta $ where $\delta \in \{0,1\}$; in particular:
        \item for every $x, m \in \N$, $l_n(x^m)= ml_n(x) + \delta$ with $\delta \in \{0,\dots, m-1\}$.
    \end{itemize}
\end{rmrk}

Being in the same $\asymp$-class has also the following nice equivalent formulations; condition (3) of the following Proposition will come particularly handy in some  proofs.

\begin{prop}\label{3 stessa classe sse log a distanza finita}
Let $\alpha,\beta\in \s\N$. The following are equivalent:
\begin{enumerate}
\item $\alpha \asymp \beta$;
\item $\frac{\alpha}{\beta}\in \s \Q$ is finite and not infinitesimal;
\item for all $n\in\N, n>1$, $l_{n}(\alpha) - l_{n}(\beta)\in \Z$;
\item $l_2(\alpha) - l_2(\beta)\in \Z$.
\end{enumerate}
\end{prop}

\begin{proof}
$(1) \Rightarrow (2)$ Suppose $\frac{\beta}{n}< \alpha < n\beta$ for some $n\in \N$. Then $\frac{1}{n } < \frac{\alpha}{\beta}< n$ which means that $\frac{\alpha}{\beta}\in \s\Q$ is finite and not infinitesimal. 

$(2) \Rightarrow (3)$ Suppose $l_n(\alpha) - l_n(\beta)$ is (say, positive) infinite. By the definition of integer logarithm $\alpha \geq n^{l_n(\alpha)}$ and $\beta\leq n^{l_n(\beta) + 1}$. So we get $\frac{\alpha}{\beta}\geq \frac{n^{l_n(\alpha)}}{n^{l_n(\beta) + 1}}=n^{l_n(\alpha) - l_n (\beta) -1}$ which is infinite since the exponent is infinite: so $(2)$ must be false.

$(3)\Rightarrow (1)$ We know that 
\begin{align*}
    n^{l_n(\alpha)}\leq \alpha < n^{l_n(\alpha) +1}, \\
    n^{l_n(\beta)}\leq \beta < n^{l_n(\beta) +1}.
\end{align*}
Suppose $l_n(\alpha ) - l_n(\beta)= k\in \N$. So we can write:
$$\beta\cdot n^{k-1} < n^{l_n(\beta)+1 + k-1}=n^{l_n(\alpha)}\leq  \alpha < n^{l_n(\alpha) +1}=n^{l_n(\beta) + k + 1}\leq \beta \cdot n^{k+1}$$ which means $\alpha \asymp \beta$.

$(3 ) \Leftrightarrow (4)$ It follows immediately from the base change formula for logarithms.
\end{proof}

 Due to the last equivalence in the proposition above, in what follows we will not stress on the base we will choose: in fact, we will usually work in base $2$, and we will indicate integer logarithm in base $2$ with $l$ instead of $l_2$.

\begin{rmrk}
    If $\alpha \sim \beta$ are in the same Archimedean class, Lemma \ref{2 proprietà generali sim}.1 and Proposition \ref{3 stessa classe sse log a distanza finita} ensure that $l_n(\alpha) - l_n (\beta)=0$.
\end{rmrk}

\begin{rmrk}
    We recall here that for every finite $\rho\in\s\R$ there exists a unique $r=st(\rho)\in \R$, called the standard part $st(\rho)$, such that $\rho - r$ is infinitesimal. Point $(2)$ of the Proposition above implies that $\alpha\asymp \beta$ if and only if $st(\frac{\alpha}{\beta})>0$.
\end{rmrk}

Other easy-but-useful properties of $\asymp$ are listed below.

\begin{lemma}\label{3 Lemma robette}

    Let $n\in\N$, $\alpha,\beta,\gamma,\delta \in \s\N$: 
    \begin{enumerate}
        \item if $\alpha \gg \beta$ then $\alpha + \beta \asymp \alpha$;
        \item if $\alpha \asymp \beta $ and $\gamma \asymp \delta $ then $\alpha +\gamma\asymp \beta + \delta $;
        \item if $\alpha \asymp \beta $ and $\gamma \asymp \delta$ then $\alpha \cdot\gamma\asymp \beta \cdot \delta $; in particular, $\alpha \asymp \beta $ if and only if $\alpha \gamma \asymp \beta \gamma$;
        \item $\alpha \asymp \beta $ if and only if $\alpha^n \asymp \beta^n$.
    \end{enumerate}
    
\end{lemma}

\begin{proof}
    $(1)$ This is obvious ($st(\frac{\alpha + \beta}{\alpha})=1$).

    $(2)$ It is enough to distinguish the case $\alpha \gg \gamma$ and the case $\alpha \asymp \gamma$: in the first one point $(1)$ above ensures that $\alpha + \gamma \asymp \alpha \asymp \beta \asymp \beta + \delta $; for the second one we know there are $n,m\in \N$ such that $\alpha<n \beta$ and $\gamma < m \delta$: if $N=\max(n,m)$ then $\alpha + \gamma < N (\beta  + \delta)$. The other inequality is proven similarly.

    $(3)$ If $\frac{1}{n}\alpha <\beta <n\alpha$ and $\frac{1}{m}\gamma<\delta <m\gamma$ then $\frac{1}{nm}\alpha\cdot \gamma<\beta\cdot\delta<nm\ \alpha \cdot \gamma$.

    $(4)$ We just observe that  $\frac{1}{k}\alpha <\beta <k\alpha$ if and only if $\frac{1}{k^n}\alpha^n <\beta ^n<k^n\alpha^n$.
\end{proof}

\begin{prop}\label{2 remark sulla Rado condition degli esponenti}
    Let $h,k\in\N$, let $n_{1},\dots,n_{h},m_{1},\dots,m_{k}\in\N$ and, for all $i\leq h,j\leq k$ let $\alpha_i\sim\beta_j$. If $\alpha_1^{n_1}\dots \alpha_h^{n_h}\asymp \beta_1^{m_1}\dots \beta_k^{m_k}$ then $\{n_1,\dots, n_h\}\cup\{-m_{1}, \dots, -m_{k}\}$ satisfies Rado condition, i.e. $\sum_{i\in I}n_{i}=\sum_{j\in J}m_{j}$ for some nonempty $I\subseteq\{1,\dots,h\},J\subseteq\{1\dots,k\}$.
\end{prop}

\begin{proof} By contrast: if $\alpha_1^{n_1}\dots \alpha_h^{n_h}\asymp \beta_1^{m_1}\dots \beta_k^{m_k}$, by taking the logarithms and applying Proposition \ref{3 stessa classe sse log a distanza finita}.(3) we arrive to $\sum_{i=1}^{h}n_{i}l\left(\alpha_{i}\right)=t+\sum_{j=1}^{k}m_{j}l\left(\beta_{j}\right)$ for some $t\in\Z$. By Lemma \ref{2 proprietà generali sim}.(1) $l\left(\alpha_{1}\right)\sim\dots\sim l\left(\beta_{l}\right)$, hence the equation $\sum_{i=1}^{h}n_{i}x_{i}=t+\sum_{j=1}^{k}m_{j}y_{j}$ is PR by Theorem \ref{2 Teo fondamentare PR sse esistono soluzioni equiv}. Our thesis follows as, by the inhomogeneos Rado condition (\cite{Rado}), this PR entails that a sum of the coefficients of the linear equation must be zero. \end{proof}

As said, we will use $\asymp$ to deal with asymptotic PR. Definition \ref{defn:asymPR}, read through the lens of Theorem \ref{2 Teo fondamentare PR sse esistono soluzioni equiv}, allows us to prove Theorem (C) from the Introduction.

\begin{theorem}\label{thm:nscharasympPR}
 Let $P\left(x_{1},\dots,x_{n}\right)\in \Z\left[x_{1},\dots,x_{n}\right]$, and let $I_{1}\uplus\dots \uplus I_{s}=\{1,\dots,n\}$. The following are equivalent:
    \begin{enumerate}
        \item $P(x_1,\dots, x_n)= 0 $ is asymptotically PR in $I_{1},\dots,I_{s}$;
        \item There exists $(\alpha_1,\dots, \alpha_n)\in \s\N^n$ such that:
        \begin{itemize}
        \item[(i)]  for all $i,j\leq n,\  \alpha_i\sim\alpha_j$;
            \item[(ii)] for all $r\leq s$, for all $i,j\in I_{r}$ $\alpha_{i}\asymp \alpha_{j}$;
            \item[(iii)] for all $1\leq r < t \leq s$, for all $i\in I_{r},j\in I_{t}$ $\alpha_{i}\gg\alpha_{j}$;
            \item[(iv)] $P\left(\alpha_{1},\dots,\alpha_{n}\right)=0$.
        \end{itemize}
    \end{enumerate}  
     
\end{theorem}

\begin{proof}
    For the sake of simplicity, let us suppose that $s=2$, the general argument being similar. Also, after rearranging the indices, we may assume that $I_1=\{1,\dots, k\}$ and $I_2=\{k+1,\dots, n\}$.
    
    $(1)\Rightarrow (2)$ For $N\in \N$ and $A\sub \N$, let $\Gamma(N, A)$ be the set 
    $$\{\Vec{x}=(x_1,\dots, x_n )\in \N^n\ \mid P(\Vec{x})=0\ \land \ \ostar_N(\Vec{x}) \ \land \ (\Vec{x}\in A^n\ \lor\ \Vec{x}\in (A^c)^n) \}$$ where $\ostar_N(\Vec{x})$ means that $\Vec{x}=(x_{1},\dots,x_{n})$ satisfies conditions $(1),(2)$ of Definition \ref{defn:asymPR} with $N$ as parameter. 
    It follows from $P$ being asymptotically PR that the family $\{\Gamma(N,A)\mid N\in \N \ , A\sub \N \}$ has the Finite Intersection Property. By saturation, $$\bigcap_{N\in \N, A\sub \N}\s\Gamma(N,A)\not=\emptyset.$$
    If $(\alpha_1,\dots, \alpha_n)$ is in the above intersection then:\begin{itemize}
        \item $\alpha_1\sim\dots\sim\alpha_n$, as for every set $A$, all $\alpha_{i}$'s are either in $\s A$ or in $\s A^c$, hence $(i)$ holds;
        \item $\alpha_1\asymp \dots \asymp \alpha_k\gg\alpha_{k+1}\asymp \dots \asymp \alpha_n$, as condition $\ostar_N$ holds for all $N\in\N$, hence $(ii),(iii)$ holds;
        \item $P\left(\alpha_{1},\dots,\alpha_{n}\right)=0$ by construction, hence $(iv)$ holds.
    \end{itemize}

    $(2)\Rightarrow (1)$ Fix a finite coloring $\N=C_1\uplus\dots \uplus C_r$ and $N\in\N$. We want to find a monochromatic solution $(x_1,\dots, x_n)$ such that:
     \begin{enumerate}
        \item for all $i,j\in I_{l}$ ($l=1,2$) $|\frac{x_i}{x_j}-1|<\frac{1}{N}$;
        \item $i\in I_{1}, j\in I_{2}$ $s_{i}> Ns_{j}$.
    \end{enumerate}
    Let $\alpha_1,\dots, \alpha_n$ be given by the hypothesis $(2)$. In $\s\N$ the following hold: 
    \begin{itemize}
        \item $\exists i\in\{1,\dots, r\}$ such that $(\alpha_1,\dots, \alpha_n)\in \s C_i $ (they are equivalent);
        \item if $i,j\in I_{1}$ or $i,j\in I_{2}$, $\vert \frac{\alpha_i}{\alpha_j}-1\vert$ is infinitesimal: in particular it will be smaller than $\frac{1}{N}$;
        \item if $i\in I_{1}, j\in I_{2}$ then $\frac{\alpha_i}{\alpha_j}$ is infinite: in particular, bigger than $N$;
        \item $P\left(\alpha_{1},\dots,\alpha_{n}\right)=0$.
    \end{itemize}
    We conclude by downward transfer applied to the formula expressing the existence in $C_{i}$ of points satisfying the desired asymptotic relations with bounds $N,\frac{1}{N}$.
\end{proof}

Theorem \ref{thm:nscharasympPR} provides a simple way to handle asymptotics. To prove this claim, we start by providing immediate proofs of Theorem \ref{thmasympclas} and Theorem \ref{thm char linear asymp}.

\begin{proof}[Second proof of Theorem \ref{thmasympclas}] As $P\left(x_{1},\dots,x_{n}\right)=0$ is PR, by Theorem \ref{2 Teo fondamentare PR sse esistono soluzioni equiv} there are $\alpha_{1}\sim\dots\sim\alpha_{n}$ with $P\left(\alpha_{1},\dots,\alpha_{n}\right)=0$. Now divide $\{\alpha_{1},\dots,\alpha_{n}\}$ in $\asymp$-equivalence classes $A_{1}\gg \dots\gg A_{s}$, where $A_{i}\gg A_{j}$ means that for all $\alpha\in A_{i},\beta\in A_{j}$ $\alpha\gg\beta$. For $i\leq s$ let $I_{i}:=\{j\mid \alpha_{j}\in A_{i}\}$. By Theorem \ref{thm:nscharasympPR}, we deduce that $P(x_1,\dots, x_n)= 0 $ is asymptotically PR in $I_{1},\dots,I_{s}$.\end{proof}

\begin{proof}[Second proof of Theorem \ref{thm char linear asymp}] It is well known (see e.g. the original proof of Rado's Theorem in \cite{Rado}) that, given a linear equation satisfying Rado's condition, there is a sufficiently large $n\in\N$ such that the equation admits a solution inside any set of the form $\{a,b,a+ib\mid i\leq n\}$. By Rado's Theorem for linear systems, for fixed $n$ such configurations are PR. In nonstandard terms, this means that there exists $\alpha,\beta\in\s\N$ such that $\alpha\sim\beta\sim\alpha+i\beta$ for all $i\leq n$. We conclude by observing $\{\alpha,\beta,\alpha+i\beta\mid i\leq n\}$ can be split in at most\footnote{Actually, when $n\geq 3$ this construction always give exactly two classes, see \cite[Proposition 5.14]{di2025ramsey}.} two Archimedean classes, one including elements $\asymp$-equivalent to $\alpha$, and one including those $\asymp$-equivalent to $\beta$. \end{proof}

Theorem  \ref{thm:nscharasympPR} tells us that to study asymptotic PR from the nonstandard point of view it is important to understand how $\asymp$ and $\sim$ interact. Some preliminary results in this direction appeared in \cite{di2025ramsey}, where asymptotic relations were used to study the PR of certain infinitary configurations related to algebraic versions of Ramsey's Theorem. In particular, the following results were proven.

\begin{theorem}[\cite{di2025ramsey}, Lemma 5.5.1]\label{2 stessa classe e sim allora st=1}

    If $\alpha \sim \beta $ and $\alpha \asymp \beta$ then $st(\frac{\alpha}{\beta})= 1$. 
    
\end{theorem}

\begin{coroll}[\cite{di2025ramsey}, Lemma 5.5.2]\label{3 corollario x^n asymp y^m allora n=m}

    If $\alpha \sim \beta$ and $\alpha^n \asymp \beta^m$ then $n = m$ and $\alpha \asymp \beta$.
    
\end{coroll}

We take this opportunity to provide a simplified proof of the above Corollary.

\begin{proof}[Proof of Corollary \ref{3 corollario x^n asymp y^m allora n=m}.] By \ref{3 stessa classe sse log a distanza finita}.(3), by applying $l$ to $\alpha^{n}\asymp\beta^{n}$ we get $$l(\alpha^n )= l(\beta^m) + c$$ for some $c\in\Z$, hence $nl(\alpha)=ml(\beta)+d$ for some $d\in\Z$. Notice that, in particular, $l(\alpha )\asymp l(\beta)$; since $l(\alpha)\sim l(\beta)$, Theorem \ref{2 stessa classe e sim allora st=1} implies that $n=m$ . We conclude by Lemma \ref{3 Lemma robette}.(4). \end{proof}

\begin{rmrk}\label{3 rmrk 2 variabili}
\begin{enumerate}
\item Corollary \ref{3 corollario x^n asymp y^m allora n=m} holds, more in general, when $P(\alpha)\asymp Q(\beta)$: $P(\alpha)$ is in the same Archimedean class of its leading term, and the same holds for $Q(\beta)$, hence it must be $\deg(P)=\deg(Q) $.
\item Corollary \ref{3 corollario x^n asymp y^m allora n=m} generalizes trivially also to monomials\footnote{In all this paper, by monomial we mean a primitive monomial, i.e. we always assume the coefficients of the monomials to be equal to 1.} consisting of exactly two variables: in fact, if $\alpha ^n \beta^m\asymp \alpha ^{n'}\beta^{m'}$ then, clearly, either $n=n^{\prime}, m=m^{\prime}$ or (without loss of generality) $n>n^{\prime}, m<m^{\prime}$. In the latter case, by Lemma \ref{3 Lemma robette}.(3) we deduce that $\alpha^{n-n^{\prime}} \asymp \beta^{m^{\prime}-m}$, and Corollary \ref{3 corollario x^n asymp y^m allora n=m} tells us that $n+m=n' + m'$. This is a special case of Proposition \ref{2 remark sulla Rado condition degli esponenti}.

    \end{enumerate}
\end{rmrk}

\begin{rmrk}\label{1 rmrk cose tornano in R}
    The same result is trivially true in $\s\R$, since the proof of Lemma 5.5.1 in \cite{di2025ramsey} works, mutatis mutandis.
\end{rmrk}

In its apparent simplicity, Theorem \ref{2 stessa classe e sim allora st=1} was elemental in \cite{di2025ramsey} to characterize the so-called Ramsey PR of several configurations. Here we want to strengthen it, and Corollary \ref{3 corollario x^n asymp y^m allora n=m} as well, to monomials involving multiple variables. To do so, we will need some preliminary facts and definitions.

\begin{defn}
    We define the \textbf{standard }head (in base $p$) of $\alpha$ as $st\left(\frac{\alpha}{p^{l_p(\alpha)}}\right)$, and denote it as $sh_p(\alpha)$. For the sake of easiness we will sometimes omit the subscript $p$. 
\end{defn}

E.g., if $\alpha=2^\beta$ then its standard head in base 2 is $1$. In general, write $\alpha $ in base $p$ and put a comma after the first digit: the standard part of the obtained number will be the standard head. It follows from the definition of $st$ that the standard head $sh_p(\alpha)$ is the unique $r\in \R$ such that $\alpha=(r + \varepsilon)p^{l_p(\alpha)}$ for $\varepsilon\in \s\R$ infinitesimal.


\begin{defn}
    If $\alpha\in \s\N$ is infinite, we say that $\beta$ is an approximation of $\alpha $ if $\vert\alpha - \beta\vert \ll \alpha$. 
\end{defn}

\begin{rmrk}\label{1 rmrk le teste approssimano il numero}
    $sh_p(\alpha)p^{l_p(\alpha)}$ and $\lfloor sh_p(\alpha)p^{l_p(\alpha)} \rfloor$ are approximations of $\alpha$.
\end{rmrk}

\begin{prop}
    If $\beta_1,\beta_2$ are approximations of $\alpha_1,\alpha_2$ respectively, then $\beta_1\beta_2$ is an approximation of $\alpha_1\alpha_2$.
\end{prop}

\begin{proof}
    Write $\alpha_i=\beta_i + \varepsilon_i$ with $\vert\varepsilon_i\vert\ll \alpha_i$. Then $$\alpha_1\alpha_2 - \beta_1\beta_2 = \beta_1\varepsilon_2 + \varepsilon_1 \beta_2 + \varepsilon_1\varepsilon_2\ll\beta_{1}\beta_{2},$$
    as $\beta_1\varepsilon_2, \varepsilon_1 \beta_2, \varepsilon_1\varepsilon_2\ll\beta_{1}\beta_{2}$.
\end{proof}

\begin{coroll}\label{2 Prendere approssimazioni non cambia i monomi}
    Let $\beta_1,\dots, \beta_n$ be approximations of $\alpha_1, \dots, \alpha_n$ respectively. If $M(x_1,\dots, x_n)$ is a monomial, then $M(\beta_1, \dots, \beta_n)$ approximates $ M(\alpha_1,\dots, \alpha_n)$.
\end{coroll}

\begin{rmrk}
    This is not true in general for polynomials, as cancellations may occur when multiple monomials are involved: i.e, take $P(x,y):=x-y$ and evaluate it in $(\alpha+\varepsilon,\alpha)$ (for $0\ll\varepsilon\ll\alpha$) and $(\alpha,\alpha)$.
\end{rmrk}

A major feature of the standard head is that it is invariant via $\sim$-equivalence. 

\begin{prop}
    If $\alpha\sim \beta$ then $sh(\alpha)=sh(\beta)$.
\end{prop}

\begin{proof}
  We just have to observe that, if $F:\N\to [1,p)$ is the function that maps $x$ to $\frac{x}{p^{l_p(x)}}$, then $F(\alpha)\sim F(\beta)$ by Lemma \ref{2 proprietà generali sim}.(1). As $F(\alpha), F(\beta)$ are $\sim$-equivalent finite hyperreals, they have same standard part. We conclude as, by construction, $sh_p(\alpha)=st(F(\alpha))$. \end{proof}

\begin{prop}
    Let $\alpha\leq\beta$ in $\s\R$. Then:
    \begin{itemize}
        \item If $\alpha\asymp \beta$, then $sh(\alpha + \beta) = \left(sh(\beta ) + p^{l_p(\alpha) - l_p(\beta)} sh(\alpha)\right)p^{a}$, where $a\in\{0,1\}$;
        \item If $\alpha\ll \beta$ then $sh(\alpha + \beta) = sh(\beta)$;
        \item $sh(\alpha\beta) =p^{a}sh(\alpha) sh(\beta),$ where $a\in\{0,1\}$.
    \end{itemize}
    
\end{prop}

\begin{proof}
    The claims follow easily by writing $\alpha=(sh(\alpha) + \varepsilon_1)p^{l_p(\alpha)}$, $\beta=(sh(\beta) + \varepsilon_2)p^{l_p(\beta)}$ and using Remark \ref{basic prop log}. 
\end{proof}

The standard head $sh(\alpha)$ plays a fundamental role in understanding which polynomials can be PR and which cannot. It is a concept implicitly introduced in \cite[Lemma 3.11]{di2018ramsey}, and related to the maximal Rado condition introduced in \cite[Definition 2.16]{BarrettLupiniMoreira}. Let us give a couple of examples.

\begin{example} Let us consider the equation $xy=z$, which is partition regular by the multiplicative version of Schur's Theorem. Fix the base $p$, and let $\alpha\sim\beta\sim\gamma$ with $\alpha\beta=\gamma$. As $sh(\alpha)=sh(\beta)=sh(\gamma)$ and $sh(\alpha\beta)=p^{a}sh(\alpha)sh(\beta)$ for $a\in\{0,1\}$, we deduce that necessarily $sh(\alpha)\in\{1,p\}$. Namely, $\sim$-equivalent solutions to $xy=z$ must be very close to powers of $p$ for all finite $p$. Explicitly, this means that their base $p$ expansions must have either the form:
\begin{itemize}
    \item a 1 followed by an infinite amount of $0$'s (that could be eventually followed by other digits), or
    \item an infinite list of $(p-1)$'s (that could again be eventually followed by other digits).
\end{itemize}\end{example}

\begin{example} Let us now consider the equation $xy=2z$. This equation is PR (see e.g. \cite[Example 2.17]{di2018ramsey}). Fix the base $p>2$, and let $\alpha\sim\beta\sim\gamma$ with $\alpha\beta=2\gamma$. As $sh(\alpha)=sh(\beta)=sh(\gamma)$, $sh(\alpha\beta)=p^{a}sh(\alpha)sh(\beta)$ for $a\in\{0,1\}$, and $sh(2\gamma)=2sh(\gamma)p^{b}$, for $b\in\{0,1\}$, we deduce that necessarily $sh(\alpha)=2p^{b-a}$, and we could argue similarly as in the previous example regarding the digits of $\alpha$.\end{example}

The examples above tells us how to compute $sh$ in the case of equalities between monomials. This can be extended to $\asymp$ as follows.

\begin{theorem}\label{2 teorema parti standard monomi archimedei} 
Let $\alpha_{1}\sim\dots\sim\alpha_{n}$, let $p\in\N, p\geq 2$ and let $\sigma$ be the standard head of $\alpha_{1},\dots,\alpha_{n}$ in base $p$. Let $M_{1}\left(x_{1},\dots,x_{n}\right),\dots,M_{k}\left(x_{1},\dots,x_{n}\right)$ be monomials, with $M_{i}=\prod_{j=1}^{n}x_{j}^{d_{i,j}}$, and let $d_{i}=\sum_{j=1}^{n}d_{i,j}=deg\left(M_{i}\right)$ for all $i\leq k$. Assume that, for all $i,j\leq k$, $M_{i}\left(\alpha_{1},\dots,\alpha_{n}\right)\asymp M_{j}\left(\alpha_{1},\dots,\alpha_{n}\right)$. 

Then there exists $m\in\Z$ such that:

\begin{enumerate}
        \item for all $i,j\leq n$, $\sum_{k=1}^{n} \left( d_{i,k}-d_{j,k})l_{p}(\alpha_{k}\right)=\left(d_{i}-d_{j}\right)m$;
        \item for all $i,j\leq n$ $st\left(\frac{M_{i}\left(\alpha_{1},\dots,\alpha_{n}\right)}{M_{j}\left(\alpha_{1},\dots,\alpha_{n}\right)}\right)=\left(p^{m}\sigma\right)^{d_{i}-d_{j}}$.
\end{enumerate}
\end{theorem}

\begin{proof}
    $(1)$ Consider any two monomials, say $M_1$ and $M_2$. Since they are in the same Archimedean class, when we apply $l_p$ we find:
    $$\sum_{j=1}^n d_{1,j} l_{p}(\alpha_j)=\sum_{j=1}^nd_{2,j}l_{p}(\alpha_j) + k_{1,2},$$
which forces $\sum_{j=1}^n (d_{1,j} - d_{2,j})l_{p}(\alpha_j)=k_{1,2}$.   So, we find that $l_{p}(\alpha_1), \dots, l_{p}(\alpha_n)$ solve the linear system
    \begin{equation}\label{2 Sistema}
        \begin{cases}
            \sum_{h=1}^{n} \left(d_{1,h}-d_{2,h}\right)x_{h}=k_{1,2}\\
            \vdots \\
            \sum_{h=1}^n(d_{i,h} - d_{j,h})x_h=k_{i,j}\ \ (i<j)\\
            \vdots \\
            \sum_{j=1}^{n} \left(d_{n-1,h}-d_{n,h}\right)x_{h}=k_{n-1,n}
        \end{cases}
    \end{equation}
    Because of Theorem \ref{2 Teo fondamentare PR sse esistono soluzioni equiv}, System (\ref{2 Sistema}) is PR; from the inhomogeneous Rado's theorem (\cite{Rado}), this is possible if and only if it has a constant solution, and the associated homogeneous system is PR. In particular there exists $m\in \Z$ such that for every $i,j\leq n$  $\sum_{h=1}^n (d_{i, h} - d_{j,h})m= k_{i,j}$. Hence
       $$ (d_i - d_j)m=k_{i,j}=\sum_{h=1}^n(d_{i,h} - d_{j,h})l_{p}(\alpha_h).$$

    $(2)$ By applying Corollary \ref{2 Prendere approssimazioni non cambia i monomi} and Remark \ref{1 rmrk le teste approssimano il numero}, we can substitute $\sigma\cdot p^{l_p(\alpha_i)} $ to $\alpha_i$. Then: 
    $$st\left(\frac{M_{i}(\alpha_1,\dots, \alpha_n)}{M_{j} (\alpha_1,\dots,\alpha_n)}\right)= st\left(\frac{M_i(\sigma p^{l_p(\alpha_1)}, \dots, \sigma p^{l_p(\alpha_n)})}{M_j(\sigma p^{l_p(\alpha_1)}, \dots, \sigma p^{l_p(\alpha_n)})}\right).$$
    An easy computation shows that the right hand side above is exactly $(p^m\sigma )^{d_i - d_j}$. 
\end{proof}

\begin{rmrk}\label{2 Remark monomi omogenei} In the notations of Theorem \ref{2 teorema parti standard monomi archimedei}, when the monomials $M_{i}$ are homogeneous we get that $st\left(\frac{M_{i}}{M_{j}}\right)=1$, which has  Theorem \ref{2 stessa classe e sim allora st=1} as a particular case.    
\end{rmrk}

The case of monomials being in the same Archimedean class arises when working with partition regular equations. In such a case, we obtain the following result.

\begin{prop}\label{1 Teorema sulle teste}
    Let $t\in\N$ and, for $i\leq t$, let $M_{i}\left(x_{1},\dots,x_{n}\right)$ be a monomial. Suppose that $P\left(x_1,\dots,x_n\right):=\sum_{i=1}^{t} M_{i}\left(x_{1},\dots,x_{n}\right)=0$ is PR, and let $\alpha_1\sim \dots \sim \alpha_n\in\s\N$ be such that $P\left(\alpha_{1},\dots,\alpha_{n}\right)=0$. Let $\sigma=sh_p(\alpha_{1})$. Assume that $M_1(\alpha_1,\dots, \alpha_n)\asymp \dots \asymp M_k(\alpha_1,\dots, \alpha_n)$ are the largest monomials with respect to the Archimedean equivalence, and define $$P_{max}(x)=\sum_{i=1}^k M_i(x,\dots, x).$$
    Then there exists $m\in \Z$ such that $P_{max}(p^m\sigma)=0$. 
\end{prop}

\begin{proof} By Theorem \ref{2 teorema parti standard monomi archimedei}, divide $P\left(\alpha_{1},\dots,\alpha_{n}\right)=0$ by $M_1(\alpha_1,\dots, \alpha_n)$ and take the standard part. This gives the conclusion.\end{proof}

\begin{coroll}
    If $\alpha$ is such that there exists a nonzero polynomial $P\left(x_{1},\dots,x_{n}\right)$ such that the equation $P\left(x_{1},\dots,x_{n}\right)=0$ can be solved by nonstandard points $\sim$-equivalent to $\alpha$, then $sh(\alpha)$ must be algebraic.
\end{coroll}

\begin{coroll} It $\sum_{i=1}^{n}c_{i}x_{i}=0$ is asymptotically PR in $I_{1},\dots,I_{s}$, then $\sum_{i\in I_{1}} c_{i}=0$.\end{coroll}

Notice that Proposition \ref{1 Teorema sulle teste} can be seen as a (slightly more explicit) reformulation of Proposition 3.1 in \cite{BarrettLupiniMoreira}; this should not come as a surprise, as Rado functionals, introduced in \cite{BarrettLupiniMoreira} and developed further in \cite{LBARadoFunctionals}, share with asymptotic PR the idea of partitioning solutions to equations in terms of some notion of largeness, which leads to a good control on the ``large pieces'' of the equations.

We conclude this section by highlighting some general facts about standard heads.

\begin{rmrk}
    The standard head $sh_p(\alpha)$ depends highly on $p$. For example, it follows from Dirichlet's approximation theorem that a power of $2$ can start with any fixed finite sequence of digits in base $10$. By transfer, there exists $\alpha \in \s\N$ such that $2^\alpha$ and $\pi$ have the first $N$ digits equal in base $10$, for $N$ infinite. In particular $sh_{10}(2^\alpha)=\pi$  while $sh_2(2^\alpha)=1$.
\end{rmrk}

\begin{rmrk} In contrast with the Remark \ref{2 Remark monomi omogenei}, it is not true in general that Archimedean equivalence between monomials forces their ratio to have standard part equal to $1$. In fact, for every real number $r\in (0,+\infty)$ and for every $n\in\N$ there exist $\alpha\sim\beta\sim \gamma\in \s\N$ such that $st(\frac{\alpha\beta^n}{\gamma})=r$.

This can be shown as follows: by the multiplicative version of Rado's Theorem (and of van der Waerden's Theorem as well) and by Theorem \ref{2 Teo fondamentare PR sse esistono soluzioni equiv} there exist $\eta\sim\mu\sim\xi$ such that $\eta\mu^n=\xi$. Let $s=\sqrt[n]{r}$ and 
    \begin{align*}
        \alpha= \lfloor s\cdot\eta \rfloor,\ \
        \beta=  \lfloor s\cdot\mu \rfloor, \ \
        \gamma= \lfloor s\cdot\xi \rfloor.
    \end{align*}

    Since $\alpha,\beta,\gamma$ are the image of equivalent numbers through a standard function (namely $n\mapsto \lfloor s\cdot n \rfloor$), they are equivalent, and  $st\left(\frac{\alpha\beta^n}{\gamma}\right)=r$.
\end{rmrk}

\begin{rmrk} There exists $\alpha\in\s\N$ such that for all $n\geq 2\in\N$, for all $\alpha_{1}<\dots<\alpha_{n}$ equivalent to $\alpha$ and for all monomials $M_{1}\left(x_{1},\dots,x_{n}\right)\neq M_{2}\left(x_{1},\dots,x_{n}\right)$ one has that $M_{1}\left(\alpha_{1},\dots,\alpha_{n}\right)\not\asymp M_{2}\left(\alpha_{1},\dots,\alpha_{n}\right)$. 

For example, let $\gamma$ be infinite and let $\alpha=2^{\gamma^\gamma}$. So, if a number is equivalent to $\alpha$ it must be of the form $2^{\delta^\delta}$ for some $\delta\sim\gamma$ due to Lemma \ref{2 proprietà generali sim}.(3). Take $\alpha_1=2^{\gamma_1^{\gamma_1}}>\dots > \alpha_n=2^{\gamma_n^{\gamma_n}}$ equivalent to $\alpha$ (and hence $\gamma_1\sim\dots\sim\gamma_n\sim\gamma$). Assume that $M_1(x_1,\dots, x_n)=\prod_{i=1}^n x^{c_i}$ and $M_2(x_1,\dots, x_n)=\prod_{i=1}^n x^{d_i}$. If then $M_1(\alpha_1,\dots, \alpha_n)\asymp M_2(\alpha_1,\dots, \alpha_n)$, apply $l_2$ to find $$c_1 \gamma_1^{\gamma_1} +\dots + c_n \gamma_n^{\gamma_n}=d_1 \gamma_1^{\gamma_1} + \dots + d_n \gamma_n^{\gamma_n} + m$$ for some finite $m$. Notice that $\gamma_1^{\gamma_1}\gg\dots\gg\gamma_n^{\gamma_n}$: so it must be $c_1=d_1$, and inductively, $c_i=d_i$.   
\end{rmrk}

\section{Consequences for the partition regularity of equations}

\subsection{Old results from a new lens}\label{sec OldRes}

Archimedean classes and asymptotics provide a simple way to obtain necessary conditions for the PR of equations. In this section, we show how all main known results in the area admit trivial proofs when approached with this method.

The idea is always the same: start with a PR equation, apply Theorem \ref{thm:nscharasympPR} (eventually, after taking some logarithm), take the $\asymp$-largest monomials and force their sum to be $\asymp$-smaller.

To clarify this idea, we start with, arguably, the most famous result about the PR of equations: Rado's Theorem for linear equations. 

\begin{theorem}[Rado] \label{2 Rado}    Let $a_{1},\dots,a_{n}\in\mathbb{Q}$. The linear equation $a_1 x_1 + \dots + a_n x_n = 0$ is PR on $\N$ if and only if there exists a nonempty set $J\sub [n]$ such that $\sum_{j\in J} a_j = 0$.    
\end{theorem}

\begin{proof}[Proof of necessity] As $a_1 x_1 + \dots + a_n x_n = 0$ is PR, by Theorem \ref{2 Teo fondamentare PR sse esistono soluzioni equiv} we have $\alpha_1\sim \dots \sim \alpha_n$ such that 
    \begin{equation}\label{eq Rado}
        \sum_{i=1}^n a_i \alpha_i=0.
    \end{equation}

    Without loss of generality, let $\alpha_{1}:=\max\{\alpha_{i}\mid i\leq n\}$, and $I_{1}=\{i\mid \alpha_{i}\asymp\alpha_1\}$. By dividing equation (\ref{eq Rado}) by $\alpha_{1}$ and taking the standard part we get
\[0=st\left(\sum_{i=1}^n a_i \frac{\alpha_i}{\alpha_{1}}\right)=\sum_{i=1}^{n} a_{i}st\left(\frac{\alpha_{i}}{\alpha_{1}}\right)=\sum_{i\in I_{1}}a_{i},\]

as $st(\frac{\alpha_i}{\alpha_1})=1$ for $i\in J_{1}$ by Theorem \ref{2 stessa classe e sim allora st=1}, and $st(\frac{\alpha_i}{\alpha_1})=0$ otherwise.
\end{proof}

The same proof, with Remark \ref{2 Remark monomi omogenei} in place of Theorem \ref{2 stessa classe e sim allora st=1}, gives the following known generalization of Rado's Theorem to homogeneous polynomials.

\begin{theorem}[{\cite[Theorem 3.5.18]{baglini2012hyperintegers}  and \cite[Corollary 3.9]{di2018ramsey}}]\label{2 Pezzo omogeneo alla Rado}
    Let $P(x_1,\dots, x_n)$ be a homogeneous PR polynomial. Then it satisfies Rado condition, i.e. a sum of the coefficients must be zero. 
\end{theorem}

The same idea can also be used to reprove a generalization of Rado's Theorem to certain nonlinear polynomials that first appeared in \cite{di2018ramsey}.

\begin{theorem}[{\cite[Theorem 3.10]{di2018ramsey}}]\label{3 Teo pezzo grande omogeneo e Rado}

    Let $P\left(x_1,\dots, x_n\right)= P_1\left(x_1\right) + \dots + P_n\left(x_n\right)$ be a polynomial in which every monomial contains a single variable. For every polynomial $P_i\left(x_i\right)$ let $a^{(i)}x^{n_i}_i$ be its monomial with higher degree. 

    If $P\left(x_1,\dots, x_n\right) = 0$ is PR then there must be a nonempty $J\sub [n]$ such that:
    \begin{itemize}
        \item if $i,j\in J$ then $n_i=n_j$, and
        \item $\sum_{j\in J} a^{(j)}=0$.
    \end{itemize}
    
\end{theorem}

\begin{proof}

    Suppose $P\left(x_1,\dots, x_n\right)=0$ is PR. Then by Theorem \ref{2 Teo fondamentare PR sse esistono soluzioni equiv} there are $\xi_1\sim\dots\sim\xi_n $ such that $P\left(\xi_1,\dots, \xi_n\right)=0$. The monomials in the largest Archimedean class in $P\left(\xi_1,\dots, \xi_n\right)$ must have the form $a^{(i)}\xi_{i}^{n_{i}}$, as for every $\alpha\in\s\N$ $\alpha^n \ll \alpha ^m $ if $n<m$. 
    
    Without loss of generality, suppose that $\xi_1^{n_1}\asymp \dots \asymp \xi_k^{n_k}$ are the $\asymp$-largest monomials. Lemma \ref{3 Lemma robette}.4 forces $n_1=\dots = n_k=:N$ and $\xi_1\asymp \dots \asymp \xi_k$. We conclude by letting $\varepsilon=P\left(\xi_{1},\dots,\xi_{n}\right)-\sum_{i=1}^{k} a^{(i)}\xi_{i}^{N}$, by dividing $0=P\left(\xi_1,\dots, \xi_n\right)$ by $\xi_1^{N}$ and by taking the standard part:
    \begin{align*}
        0&= st\left(a^{(1)}\frac{\xi_1^N}{\xi_1^N} + \dots + a^{(k)}\frac{\xi_k^N}{\xi_1^N} + \frac{\varepsilon}{\xi_1^N} \right)=\\
        & = a^{(1)}st\left(\frac{\xi_1^N}{\xi_1^N}\right) + \dots + a^{(k)}st\left(\frac{\xi_k^N}{\xi_1^N}\right) + st\left(\frac{\varepsilon }{\xi_1^N}\right)=\\
        &\stackrel{(\ref{2 stessa classe e sim allora st=1})}{=}a^{(1)}\cdot 1 + \dots + a^{(k)}\cdot 1 + st\left(\frac{\varepsilon}{\xi_1^N}\right)=\\
        &\stackrel{\varepsilon \ll \xi_1^N}{=}a^{(1)} + \dots  + a^{(k)}.\qedhere
    \end{align*}
\end{proof}

We can also generalize: indeed, the argument in the proof of Theorem \ref{3 Teo pezzo grande omogeneo e Rado} works whenever the bigger monomials have only one variable each. 

\begin{prop}

    Let $P(x_1,\dots, x_n)= a_1x_1^{d_1} + \dots + a_n x_n^{d_n} + R(x_1,\dots, x_n)$ where $\deg(R)<\min\{d_i\mid i\leq n\}$. If $P$ is PR then there must be a (nonempty) subset $J\sub [n]$ such that:
    \begin{itemize}
        \item if $i,j\in J$ then $d_i=d_j$ and
        \item $\sum_{j\in J} a_{j}=0$.
    \end{itemize}
    
\end{prop}

\begin{proof}

    We proceed as in Theorem \ref{3 Teo pezzo grande omogeneo e Rado}, looking for the largest monomials. But since $\deg(R)$ is small, these monomials must be in $\{x_1^{d_1},\dots, x_n^{d_n}\}$: indeed, order the variables with respect to their Archimedean classes. Just for the sake of simplicity, suppose that $x_1\gg x_i$ for all $i\not=1$. Then $$\prod_{i=1}^n x_i^{m_i}\ll\prod_{i=1}^n x_1 ^{m_i}=x_1^{\sum_{i=1}^n m_i}$$ and so in particular every monomial of degree $r$ in $R(x_1,\dots, x_n)$ will be smaller than $x_1^r$. Since $\deg(R)<d_i$ for all $i$, we conclude that all monomials in $R$ are smaller than $x_1^{d_1}$ and so they cannot be in the largest Archimedean class. Adapting this argument in general we get that monomials in the largest class must be in between $x_1^{d_1}, \dots, x_n ^{d_n}$. 

    We conclude as in the proof of Theorem \ref{3 Teo pezzo grande omogeneo e Rado}.
\end{proof}

Recently, there has been a rising interest in characterizing the PR of equations of the form $ax+by=cw^{m}z^{n}$. In \cite{FarhangiMagner}, a necessary condition for the PR of such equations if given; we reprove it here with our methods.

\begin{prop}[{\cite[Theorem 1.(a)]{FarhangiMagner}}]
    If $a+b\not=0 $ and $n,m>1$ the equation $ax +by = cw^mz^n $ is not PR.
\end{prop}

\begin{proof}
    Suppose the equation to be PR, and let $\alpha\sim\beta\sim\gamma\sim\delta$ be such that $a\alpha+b\beta=c\gamma^{m}\delta^{n}$. As $n,m>1$, by Proposition \ref{2 remark sulla Rado condition degli esponenti} $\alpha\asymp \beta \gg \gamma^m \delta^n$. Dividing by $\alpha$ and taking the standard part we get $a+b =0$, against our hypothesis.  
\end{proof}

\subsection{New results: partition regularity of Fermat-Catalan equations}\label{sec:new results}

Methods similar to the ones of Section \ref{sec OldRes} can be used to obtain also new results about the PR of nonlinear equations. We start with the following condition, which is an easy consequence of Proposition \ref{2 remark sulla Rado condition degli esponenti}. 

\begin{prop}
Let $P(x_1,\dots, x_n)\in \Z[x_1,\dots, x_n]$ be a polynomial such that no pair of monomials in it satisfies Rado condition on the exponents. Then $P$ is not PR.
\end{prop}

\begin{proof}
    Suppose it is, and let $\alpha_1\sim\dots\sim\alpha_n$ be a solution in $\s\N$. Due to Proposition \ref{2 remark sulla Rado condition degli esponenti}, all monomials in $P$, when evaluated in $\alpha_1,\dots,\alpha_n$, are in different Archimedean classes; in particular, their sum $P(\alpha_1,\dots,\alpha_n)$ will be in the larger one and so it cannot be 0.
\end{proof}

One of the major open problems in the area is the PR of the Pythagorean equation $x^{2}+y^{2}=z^{2}$. In \cite{di2018fermat} Di Nasso and Riggio considered generalization of this problem to Fermat-Catalan equations\footnote{Di Nasso and Riggio called these equations ``Fermat-like''; we prefer to call them Fermat-Catalan equations, following the usual number-theoretical naming.}, including $x^{n}-y^{m}=z^{k}$. Their main result, namely that such an equation is never PR if $n\notin \{m,k\}$, can be easily seen as a consequence of Theorem \ref{3 Teo pezzo grande omogeneo e Rado} or, straight in the language of asymptotic PR, of \cite[Lemma 5.5.2]{di2025ramsey}. Without loss of generality, hence, we can reduce to the case $n=m$. The following question, already asked in \cite[Section 4]{di2018fermat}, remains open.

\begin{question}\label{quest} For which $n,k$ is $x^{n}-y^{n}=z^{k}$ PR?\end{question} 

The known answers to the above question are:

\begin{itemize}
    \item yes, for $n=1$ and any $k\in\N$ (Polynomial Van der Waerden Theorem, \cite{bergelson1996polynomial});
    \item yes, for $n=2, k=1$ (\cite[Corollary 1.8]{moreira2017monochromatic});
    \item no, for $n=k\geq 3$ (Fermat's last Theorem).
\end{itemize}

We largely extend the class of negative answers to Question \ref{quest} by providing a necessary condition for the PR of a larger class of equations in the following Theorem.

\begin{theorem}\label{4 th x^n - y^n} Let $n\geq 2$, let $a,b\in\Z$ and let $P(z)\in \Z[z]$ be a polynomial of degree $k$. If $n\notin \{k,k-1\}$ then $ax^n+by^n = P(z)$ is not PR, if not trivially (i.e. with a constant solution). \end{theorem}

\begin{proof}  By exchanging $x,y$ if necessary, we can assume that  the leading coefficient $a_k$ of $P(z)$ is positive. 

First, notice that by our hypotesis on $n,k$, $a+b=0$ by Theorem \ref{3 Teo pezzo grande omogeneo e Rado}. Therefore we can write our equation as $a(x^{n}-y^{n})=P(z)$.
As we assumed this equation to be PR without constant solutions, there are $\alpha \sim \beta\sim \gamma$ in $^{\ast}\N\setminus \N$ with $\beta>\alpha$ such that $a(\beta^n - \alpha^n) =P(\gamma)=a_k\gamma ^k +\dots + a_0$. As $k\neq n$, the same argument used in Theorem \ref{3 Teo pezzo grande omogeneo e Rado} shows that  $\alpha^n \asymp \beta^n \gg \gamma^k$. In particular, $\alpha \asymp \beta$ and since they are equivalent and $\beta>\alpha$, there must exist $\varepsilon \ll\alpha$ such that $\beta=\alpha + \varepsilon$. Hence:
    $$a(\beta^n - \alpha^n) = a\left(n\alpha^{n-1}\varepsilon \stackrel{\gg}{+} \binom{n}{2} \alpha^{n-2}\varepsilon^2 +\dots\right)= a_k\gamma^k \stackrel{\gg}{+}\dots + a_0$$
    In particular, $\beta^{n}-\alpha^{n}\asymp \alpha^{n-1}\varepsilon $ and $P(\gamma)\asymp \gamma^k$, so $\alpha^{n-1}\varepsilon\asymp \gamma^{k}$. By applying $l$ we deduce that
    $$(n-1)l(\alpha) + l(\varepsilon)= kl(\gamma) + c,$$
    where $c\in \Z$ is a constant. Since $\varepsilon \leq \alpha$,
    \begin{equation}\label{3 eq fra log}
        (n-1)l(\alpha)\leq (n-1)l(\alpha) + l(\varepsilon) =kl(\gamma) + c\leq nl(\alpha).
    \end{equation}
    In particular, $l(\alpha )\asymp l(\gamma)$. As $l(\alpha)\sim l(\gamma)$, by Theorem \ref{2 stessa classe e sim allora st=1} $st\left(\frac{l(\alpha)}{l(\gamma)}\right)=1$. Dividing equation (\ref{3 eq fra log}) by $l(\alpha)$ and taking the standard part, we get $n-1\leq k\leq n$.
\end{proof}

In particular, Theorem \ref{4 th x^n - y^n} gives the following partial answer to Question \ref{quest}:

\begin{coroll}\label{3 coroll quello che sappiamo sulle Fermat}
    Assume that $x^n - y^n= z^k$ is PR. Then at least one of the following conditions holds:
    \begin{enumerate}
        \item $n=1$, or
        \item $n=2,\ k=1,2$, or 
        \item $k=n-1$.
        \end{enumerate}
\end{coroll}

So, partition regular Fermat-Catalan equations must have the form 
\begin{equation}\label{equazione di Fermat-Catalan}
    ax^n + b y^n = c z^k, \ \ k\in \{n,n-1\}.
\end{equation}

We study the cases $k=n, k=n-1$ separately.

If $k=n-1$ we know that it must be $b=-a$ by Theorem \ref{3 Teo pezzo grande omogeneo e Rado}. 

\begin{prop}\label{2 prop x^n + y^n = z^n-1 PR sse caso con i coefficenti} Let $a,c\in\N$. The following facts are equivalent:
\begin{enumerate}
    \item $x^n - y^n=z^{n-1}$ is PR;
    \item $a x^n - a y^n= cz^{n -1}$ is PR.
\end{enumerate}
\end{prop}
\begin{proof}
$(1)\Rightarrow (2)$ Under our hypothesis, by Theorem \ref{2 Teo fondamentare PR sse esistono soluzioni equiv} there are $\alpha\sim \beta \sim\gamma$ such that $\beta ^n - \alpha^n =\gamma^{n-1}$. First, let notice that $\alpha,\beta$ and $\gamma$ are divisible by every natural $m\in \N$. Indeed, since they are equivalent, $\alpha\equiv\beta\equiv\gamma $ modulo $m^{n-1}\cdot c$ for every $m$, and in particular $c\gamma^{n-1}=a(\alpha^n - \beta^n) \equiv 0$ modulo $m^{n-1}\cdot c$: so $m^{n-1}\cdot c$ must divide $c\gamma^{n-1}$ and hence $m$ divides $\gamma$ (and $\alpha$ and $\beta$). 

    But then $\alpha'=\frac{c}{a}\alpha\sim \beta' =\frac{c}{a}\beta\sim\gamma'=\frac{c}{a}\gamma$ in $\s\N$ solve $a x^n - a y^n= cz^{n -1}$, which is then PR due to Theorem \ref{2 Teo fondamentare PR sse esistono soluzioni equiv}. 

$(2)\Rightarrow (1)$ This is similar: arguing as above, we find $\alpha\sim\beta\sim\gamma\in\s\N$ such that $a\alpha^{n}-a\beta^{n}=c\gamma^{n-1}$. Reasoning similarly as above, we deduce that $\alpha,\beta,\gamma$ must are divisible by every natural $m\in \N$. Now take $\alpha^{\prime}=\frac{a}{c}\alpha\sim\beta^{\prime}=\frac{a}{c}\beta\sim\gamma^{\prime}=\frac{a}{c}\gamma$ and observe that they solve $x^{n}-y^{n}=z^{n-1}$ to conclude, again by Theorem \ref{2 Teo fondamentare PR sse esistono soluzioni equiv}.
\end{proof}

Hence, the case $k=1$ is reduced to the study on the equation $x^{n}-y^{n}=z^{n-1}$.

If $k=n$, again by Theorem \ref{3 Teo pezzo grande omogeneo e Rado} we have two cases (up to symmetries): $a + b = c$ or $a+ b =0$. 

\begin{prop} Let $a,b,c,n\in\N$. The following two facts hold:
    
    \begin{enumerate}
        \item If $a + b = c$, then $ax^{n}+by^{n}=cz^{n}$ is trivially PR;
        \item If $a + b =0$ and $n>3$, then $ax^{n}+by^{n}=cz^{n}$ is not PR.
    \end{enumerate}
\end{prop}

\begin{proof} (1) This is immediate, as the equation admits constant solutions.

(2) We will use a classical result of Darmon and Granville about the number of solutions of Fermat-Catalan equations (\cite{darmon1995equations}, Theorem 2), stating that for all integers $A,B,C$ and $p,q,r$ satisfying $\frac{1}{p} + \frac{1}{q} + \frac{1}{r}< 1$, the equation $$Ax^ p + B y^q = Cz^r$$ has only a finite number of solutions $(x,y,z)$ with $\gcd(x,y,z)=1$.

Now suppose that $ax^{n}+by^{n}=cz^{n}$ is PR, and by Theorem \ref{2 Teo fondamentare PR sse esistono soluzioni equiv} let $\alpha\sim\beta\sim\gamma $ be infinite numbers such that $a\alpha^n + b\beta^n= c\gamma^n$. Since there is only a finite number of solutions with $\gcd=1$, $(\alpha,\beta,\gamma)$ must be a multiple of one of those irreducible solutions, say $(r,s,t)\in \N^3$. So $\alpha=r \mu \sim \beta=s\mu \sim \gamma=t\mu$: but this is impossible, unless $r=s=t=1$, because of Lemma \ref{2 proprietà generali sim}.(2). As $a + b=0\neq 0$, $1$ is not a constant solution of our equations, so we reach a contradiction.  
\end{proof}

Minor modifications to the proof of Proposition \ref{4 th x^n - y^n} give the following generalization.

\begin{coroll}
    Let $P_1(z_1),\dots, P_k(z_k)$ be polynomials of different degrees, and $n\notin \{\deg(P_i),\deg(P_i)-1\mid 1\leq i\leq k\}$. Then $ax^n + by^n = \sum_{i=1}^{k} P_i(z_i)$ is not PR. 
\end{coroll}

\begin{proof} By contrast, assume the equation $ax^n + by^n = \sum P_i(z_i)$ to be PR. Again by Theorem \ref{3 Teo pezzo grande omogeneo e Rado} we deduce that $a + b =0$, and by Theorem \ref{2 Teo fondamentare PR sse esistono soluzioni equiv}, let $\alpha,\beta,\gamma_{1},\dots,\gamma_{k}$ be $\sim$-equivalent numbers such that 
\begin{equation}\label{eq facile}a(\alpha^{n}-\beta^{n})=\sum_{i=1}^{k}P_{i}\left(\gamma_{i}\right).\end{equation}

Our hypotheses ensure that $\alpha^{n}\asymp\beta^{n}\gg P_{i}\left(\gamma_{i}\right)$ for all $i\leq k$. Write $\alpha=\beta + \varepsilon$, where $\beta\gg \varepsilon$.
   
   As the degrees of the $P_{i}'s$ are all different, necessarily for $i\neq j$ $P_i(\gamma_i)\not\asymp P_j(\gamma_j)$: in fact, $P_i (\gamma_i)\asymp \gamma_i ^{\deg(P_i)}\not\asymp \gamma_j^{\deg(P_j)}\asymp P_j(\gamma_j)$ since $\gamma_i\sim \gamma_j$.

   In particular, the right hand side of equation (\ref{eq facile}) will be asymptotic to its largest monomial $\gamma_i^{\deg(P_i)}$. We can now conclude by the same argument used in the proof of Theorem \ref{4 th x^n - y^n}. 
\end{proof}

\begin{example} $x^2 - y^2 = z_1^4 + z_2^5 - 3z_3^6 $ and $x^5 - y^5 = z_1^3 - z_2^3$ are not PR. \end{example}

With a slightly more subtle argument, we can also obtain the following result, having the same spirit.

\begin{coroll}
    Let $n\geq 2$. Let $\{k_{1},\dots,k_{n}\}$ be such that no finite sum of the $k_i$ is equal to $n$ or $n-1$. Then the equation $x^n - y^n =  \prod_{i=1}^{m} z_i^{k_i}$ is not PR.
\end{coroll}

\begin{proof}
    Suppose the equation to be PR, and let $\alpha\sim \beta \sim \gamma_1\sim\dots\sim \gamma_k$ be a solution in $\s\N$. As usual, we look for monomials in the largest class.
    
    By Proposition \ref{2 remark sulla Rado condition degli esponenti}, since exponents of the right-hand side do not sum in any way to $n$, it must be $\alpha^n\asymp \beta^n\gg \prod_{i=1}^{m} \gamma_i^{k_i}$. In particular, $\alpha\asymp \beta$ and so we can write $\alpha=\beta + \varepsilon$ with $\varepsilon\ll \beta$. So:
    $$\alpha^n - \beta^n= n \alpha^{n -1} \varepsilon + \binom{n}{2}\alpha^{n-2}\varepsilon^2 + \dots = \prod_{i=1}^m \gamma_i^{k_i},$$ 
    and in particular 
    $$\alpha^{n-1}\varepsilon \leq \prod_{i=1}^m \gamma_i^{k_i} \leq \alpha^n. $$
    By taking the logarithm we arrive to
    $$(n-1)l(\alpha) + l(\varepsilon) \leq l\left(\prod_{i=1}^m \gamma_i^{k_i}\right)=\sum_{i=1}^m k_i l(\gamma_i)\leq nl(\alpha).$$
    Hence, as $l(\varepsilon)\leq l(\alpha)$, we get that \begin{equation}\label{eq:sumapr} \sum_{i=1}^m k_i l(\gamma_i)\asymp l(\alpha).\end{equation} As $l(\gamma_{i})\leq l(\alpha)$ for all $i\leq m$, by equation (\ref{eq:sumapr}) we have that $l(\gamma_{i})\asymp l(\alpha)$ for some $i\leq m$. As $l(\alpha)\sim l(\gamma_i)$, by dividing by $l(\alpha)$ and taking the standard part by Theorem \ref{2 teorema parti standard monomi archimedei} we obtain 
    $$n-1 \leq \sum_{i\in I} k_i\leq n$$ for some $I\subseteq \{1,\dots,m\}$, which is against our assumption. So the polynomial cannot be PR.
\end{proof}

\begin{example}  $x^4 - y^4 = z_1z_2$ is not PR. \end{example}

\section{Ultrafilters and Archimedean classes}

We already mentioned the natural relation between numbers in hyperextensions and ultrafilters: it should not be a surprise, then, that we can use Archimedean classes to solve problems naturally arising in the ultrafilter context. In this section we will present an example (namely, equations in $\beta\N$), and introduce the equivalent notion of Archimedean closeness for ultrafilters. But first, for the sake of completion, let us recall the link between ultrafilters and elementary extensions.

\subsection{Preliminaries: ultrafilters and hyperextensions}\label{sec:ultrafilters and hyperextensions}

 We assume the reader to be familiar with the algebra of ultrafilters: if not, we refer to the monograph \cite{hindman2011algebra}; we also refer to \cite{di2015hypernatural} for a more extensive treatment of the nonstandard approach that we summarize here. The crucial point in the nonstandard approach to ultrafilters is the following translation.

\begin{rmrk}
    Let $\U$ be an ultrafilter over $\N$ (or $\R$, say). Since $\U$ has the Finite Intersection Property, by saturation $$\bigcap_{A\in \U} \s A \not=\emptyset.$$ 
    So in $\s\N$ we can find generators for $\U$.
\end{rmrk}

The $u$-equivalence relation can be equivalently characterized in terms of ultrafilters, as seen in Section \ref{sec:nsa}. 

When working in dimension higher than one, and when dealing with operations with ultrafilters, the notion of tensor products comes out naturally. We recall it in the case we are dealing with, namely $\N^{2}$.

\begin{defn} Let $\U,\V\in\beta\N$. The tensor product of $\U$ and $\V$, denoted by $\U\otimes\V$, is the ultrafilter in $\beta\N^{2}$ defined by the following property: for all $A\subseteq\N^{2}$ 
\[A\in\U\otimes\V\Leftrightarrow\{n\in\N\mid\{m\in\N\mid(n,m)\in A\}\in\V\}\in\U.\]
\end{defn}

Tensor products can be used to define several binary operations on $\beta\N$. For example, the sum $\U\oplus\V$ of two ultrafilters can be defined as the pushforward $\overline{+}$ of the sum $+:\N^{2}\rightarrow \N$ applied to $\U\otimes\V$; namely, $\U\oplus\V=\overline{+}(\U\otimes\V)$. We refer to \cite{hindman2011algebra} for a through study of these notions. In nonstandard terms, tensor products corresponds to tensor pairs.

\begin{defn}
    We say that $(\alpha,\beta)\in \s\N^2$ is a tensor pair if $\U_{(\alpha,\beta)}=\U_{\alpha}\otimes \U_{\beta}$.
\end{defn}

Tensor pairs in $\s\N^2$ are fully characterized by the following theorem, due to Puritz.

\begin{theorem}\label{4 Puritz}{\cite[Theorem 3.4]{puritz1972skies}}
    Let $\alpha, \beta$ be infinite hypernaturals. Then
    $(\alpha, \beta )$ is a tensor pair if and only if, for every $f:\N\to \N$ such that $f(\beta)\not\in\N$, $f(\beta)>\alpha$.
\end{theorem}

Trivially, Puritz's Theorem entails the following asymptotic consequence.

\begin{coroll}\label{cor:tensimpliesarch}
    Let $(\alpha, \beta)$ be a tensor pair in $\s\N$. Then $\beta\gg \alpha$. 
\end{coroll}

It is known that operations between ultrafilters can be transformed into operations between numbers in $\s\N$ via Lemma \ref{2 proprietà generali sim}.(3) and the notion of tensor pair. For example:

\begin{theorem}\label{4 Teo su somma di ultrafiltri come somma di ipernaturali}
    Let $\U,\V\in\beta\N$, let $\alpha$ be a generator of $\U$ and $\beta$ a generator of $\V$ and assume that $(\alpha,\beta)$ is a tensor pair. Then  $\U\oplus \V$ is generated by $\alpha + \beta$.
\end{theorem}

The last property of tensors that we will use is that we can always fix a coordinate when searching for tensors, in the sense of the following consequence of Puritz's characterization. 

\begin{prop}[\cite{di2015hypernatural}, Theorem 5.12]\label{3 Prop ci sono tanti tensori}
    Let $\alpha, \beta \in \s \N$ be infinite numbers. 
    \begin{itemize}
        \item $R_{\alpha, \beta}=\{\beta'\sim\beta \mid (\alpha, \beta') \text{ is a tensor pair}\}$ is unbounded in $\s\N$;
        \item $L_{\alpha, \beta}=\{\alpha'\sim\alpha \mid (\alpha', \beta) \text{ is a tensor pair}\}$ is leftward unbounded in $\s\N\setminus \N$.
    \end{itemize} 
\end{prop}
 We refer to \cite{luperi2019tensor} for a detailed study of tensor pairs and their extensions in higher dimensions.

Finally, we will need the following known topological property of $(\beta\N,\oplus)$. Since we could not find a precise reference to this exact statement, we will provide a short proof of it here.

\begin{prop}\label{3 prop sulla chiusura delle orbite}
    Let $\U\in \beta\N$ be a non-principal ultrafilter; then its $\Z$-orbit $orb(\U)=\{n \oplus \U \mid n\in \Z\}$ is not open nor closed. 
\end{prop}

\begin{proof}
    The set $orb(\U)$ is not open because every open set in $\beta\N\setminus \N$ is uncountable. This follows easily from the observation that over a set $D$ we have $2^{2^{\vert D\vert}}$ different ultrafilters (cf. \cite{hindman2011algebra}, Section 3.6), and that given a basis set $\mathcal{O}_A$ with $A$ infinite, the map $\beta A \to \mathcal{O}_A$ extending an ultrafilter on $A$ to one on $\N$ is injective. 

    The set $orb(\U)$ is also not closed. In fact, it can be proved that $\overline{\{n \oplus \U \mid n\in \Z\}}=\beta\Z \oplus \U$, so it is enough to prove that $\beta\Z \oplus \U\not= n \oplus \U$ for every $n\in \Z$. Let $\V$ be an ultrafilter with the following property: for every $r \in \N$, 
    \begin{equation}\label{3 equazione inutile}
        R(r-1 , r)\in \V,
    \end{equation}
     
     where $R(m, r)=\{n \in \N \mid n\equiv m \text{ mod }r\}$.
    Such an ultrafilter exists, as $\{R(r-1,r)\mid r\in\N\}$ has the Finite Intersection Property. 

    We show that for every $n\in \Z$, $\V \oplus \U \not= n \oplus \U$.  Suppose, by contrast, that $m \oplus \U = \V \oplus \U$. Let $i\in [1,m+2]$ be such that $R(i, m+2)\in \U$. Then $R(i +m, m+2)=R(i, m +2) + m \in m \oplus \U =\V\oplus \U$, which means that $$A=\{n \in \N \mid R(i + m , m+2) - n \in \U\}\in \V.$$ 
    Since $R(i + m , m+2) - n=R(i + m -n, m+2)$, we deduce that $A\subseteq \{n \in \N \mid n\equiv m \text{ mod } m+2\}=R(m, m+2)\in \V$, against (\ref{3 equazione inutile}).
\end{proof}

\begin{rmrk}
    We notice that the ultrafilter $\V$ that we constructed in the proof above satisfies the following stronger property: for all $\U\in\beta\N$ $\V\oplus\U$ (and $\U\oplus\V$ as well) does not belong to $\{\U\oplus n\mid n\in\N\}$.
\end{rmrk}

\subsection{Equation in $\beta\N$ solved via Archimedean classes}\label{eq ultra}

Similar to what has been done in \cite{di2025ramsey}, asymptotics can be used to provide simple proofs of known arithmetic results about ultrafilters. To provide examples of this fact, we give here two easy proofs of well-known results first proven in \cite{Maleki} and \cite{HMS}.

\begin{theorem}[\cite{Maleki}, Theorem A]\label{5 Equazioni: pezzo grande uguale}
    Let $\U, \V, \mathcal{W}\in \beta\N$, with $\U$ non-principal. Let $a,b\in\N$. If $\mathcal{W} \oplus a\U = \V\oplus b\U$ then $a=b$.
\end{theorem}

From our nonstandard point of view, Theorem \ref{5 Equazioni: pezzo grande uguale} above can be rephrased as follow:

\begin{theorem}
    Let $a,b\in\N$ and let $\alpha,\beta,\gamma\in \s\N$ with $\alpha$ infinite and $(\beta, \alpha), (\gamma, \alpha)$ tensor pairs. Suppose that $\beta + a \alpha\sim \gamma + b\alpha$. Then $a=b$.
\end{theorem}

\begin{rmrk}
    The same is true even if we allow $a,b$ to be reals.
\end{rmrk}

\begin{proof}
    It is enough to notice that $\beta + a\alpha\asymp \gamma + b\alpha$: then, since they are equivalent, from Theorem \ref{2 stessa classe e sim allora st=1} we have 
\begin{equation*}
    1=st\left(\frac{\beta + a\alpha}{\gamma + b\alpha}\right)=\frac{a}{b}.\qedhere
\end{equation*}   
\end{proof}

The above result uses the fact that asymptotic considerations give a good control on the behaviour of the $\asymp$-large elements in a sums. The next result will show that such methods can be used also to control $\asymp$-small elements.

\begin{theorem}[\cite{HMS}, Corollary 2.7]\label{teo small ultra}
    Let $a,b\in\N$. Let $\U, \V, \mathcal{W}\in \beta\N$, with $\U$ non-principal. If $  a\U\oplus\mathcal{W} =  b\U\oplus\V$ then $a=b$.
\end{theorem}

In our nonstandard language, Theorem \ref{teo small ultra} reads as follows.

\begin{theorem}\label{4 equazioni in betaN, pezzo grande}
    Let $\alpha,\beta,\gamma\in \s\N$ with $\alpha$ infinite and $(\alpha, \beta), (\alpha, \gamma)$ tensor pairs. Suppose that $\beta + a\alpha\sim \gamma + b\alpha$: then $a=b$.
\end{theorem}

\begin{rmrk}
    It is not restrictive to assume that $\beta,\gamma $ also are infinite (the result is otherwise trivial).
\end{rmrk}

For the proof we need to work in a more general context. We already noticed, in Remark \ref{1 rmrk cose tornano in R}, that we may define the concept of Archimedean classes for hyperreals $\s\R$ as we did for $\s\N$. Also, the proof of Theorem \ref{2 stessa classe e sim allora st=1} still works in $\s\R$, so the ration of equivalent hyperreals $\alpha\sim\beta$ in the same Archimedean class must be at an infinitesimal distance from 1. 
\begin{defn}
     Let $\alpha,\beta\in \s S^1\setminus S^1$, where $S^1=\R/\Z$ is the unitary circle. We say that they are in the same Archimedean class if $\frac{1}{\vert \alpha - st(\alpha)\vert}\asymp \frac{1}{\vert \beta - st(\beta)\vert}$ in $\s\R$.
\end{defn}

\begin{rmrk}\label{rem est fraz 1}
    It is easy to check that if $\alpha, \beta $ are as above and equivalent, then $st(\frac{\alpha}{\beta})=1$.
\end{rmrk}

We also need to characterize tensor pairs in $\s S^1$. Notice that if $s\in S^1$ is a standard point and $\alpha\in \s S^1\setminus S^1$, $(s, \alpha) $ and $(\alpha, s)$ are always tensors in a trivial way. So we can suppose both coordinates to be nonstandard. We will use the following two results proven in \cite{luperi2019tensor}.

\begin{theorem}[\cite{luperi2019tensor}, Theorem 31]\label{3 Teo su come son fatti i tensori in generale}
    Let $A,B,A^{\prime},B^{\prime}$ be sets. Let $(\alpha, \beta)\in \s A \times \s B$. Then the following are equivalent:
    \begin{enumerate}
        \item $(\alpha, \beta)$ is a tensor pair;
        \item for every $f:A\to A^{\prime}$, $g:B\to B^{\prime}$ the couple $(f(\alpha), g(\beta))$ is a tensor pair in $\s A^{\prime} \times \s B^{\prime}$. 
    \end{enumerate}
\end{theorem}

\begin{rmrk}[\cite{luperi2019tensor}, Example 37]\label{3 rmrk tensori in R}
    If $\alpha, \beta$ are infinite positive numbers in $\s\R$, then $(\alpha, \beta )$ is a tensor pair if and only if for every $f:\R\to \R$ such that $f(\beta)$ infinite, $\alpha < f(\beta)$.  
\end{rmrk}

\begin{prop}\label{3 Prop tensori in S^1}
    Let $(\alpha, \beta) $ be a couple in $\s S^1\setminus S^1$. $(\alpha, \beta)$ is a tensor pair if and only if $(\frac{1}{\lvert\alpha - st(\alpha)\rvert}, \frac{1}{\lvert \beta - st(\beta)\rvert})\in \s \R \times \s \R$ is a tensor pair.\\    
    In particular, $\beta - st(\beta)$ is infinitesimal with respect to $\alpha - st(\alpha)$, i.e. $st\left(\frac{\beta - st(\beta)}{\alpha - st(\alpha)}\right)=0$.
\end{prop}

\begin{proof}
    It is enough to apply Theorem \ref{3 Teo su come son fatti i tensori in generale} with $f(x)=\frac{1}{\lvert x - st(\alpha)\rvert}$ and $g(x)=\frac{1}{\lvert x - st(\beta)\rvert}$. 
    The ``in particular'' part follows directly from Remark \ref{3 rmrk tensori in R}.\end{proof}

\begin{proof}[Proof of Theorem \ref{4 equazioni in betaN, pezzo grande}]

    Let $f$ be the composition of the multiplication by $\pi$ with the projection $\R\to S^1=\R/\Z$. $f$ is an algebraic embedding of $\Z $ in $S^1$. \\
    
    \noindent {\bfseries Claim: }if $\alpha\in \s\Z \setminus \Z$ then $f(\alpha)\in \s S^1 \setminus S^1$.\\
    
    To prove the claim, suppose there exists $\alpha$ infinite such that $f(\alpha)\in r + \Z$ with $r\in [0,1)$. In particular for every $n\in \N$ $$\s\N \models \exists m>n\ (f(m)=r).$$
    By transfer: $$\N \models \exists m>n\ (f(m)=r),$$ which is not possible because $f$ is injective.

    So, suppose $(\alpha,\beta)$ and $(\alpha, \gamma)$ are tensor pairs in $\s S^1\setminus S^1$ such that $\beta + a\alpha\sim \gamma + b \alpha$. 
    
As $\beta + a\alpha\sim \gamma + b \alpha$, we have that $st(\beta)+a\cdot st(\alpha)=st(\beta + a\alpha)=st(\gamma + b \alpha)=st(\gamma) + b\cdot st( \alpha)$. Let $\alpha^{\prime},\beta^{\prime},\gamma^{\prime}$ denote respectively $\alpha-st(\alpha),\beta-st(\beta),\gamma-st(\gamma)$. We have that 
$$\beta + a\alpha-st(\beta+a\alpha)=\beta^{\prime}+a\alpha^{\prime}\sim \gamma^{\prime}+b\alpha^{\prime}=\gamma+b\alpha-st(\gamma+b\alpha).$$
$\alpha^{\prime},\beta^{\prime},\gamma^{\prime}$ are infinitesimals by construction and, by Proposition \ref{3 Prop tensori in S^1}, $\beta^{\prime},\gamma^{\prime}\ll\alpha^{\prime}$ by definition of $\asymp $ on $\s S^1\setminus S^1$. Hence $\beta^{\prime} + a \alpha^{\prime} \asymp  \gamma^{\prime} + b\alpha^{\prime}$. As $\beta^{\prime} + a \alpha^{\prime} \asymp a \alpha^{\prime}$ and $\gamma^{\prime} + a \alpha^{\prime} \asymp b \alpha^{\prime}$, it follows that $\frac{1}{\beta^{\prime}+a\alpha^{\prime}}\asymp\frac{1}{\gamma^{\prime}+b\alpha^{\prime}}$. 

    So in $\s S^1$ we would have $\beta^{\prime} + a\alpha^{\prime}\sim \gamma^{\prime} + b\alpha^{\prime}$ and in the same Archimedean class: so by Remark \ref{rem est fraz 1}
    \begin{equation*}
    1=st\left(\frac{\beta^{\prime} + a\alpha^{\prime}}{\gamma^{\prime} + b\alpha^{\prime}}\right)=\frac{a}{b}. \qedhere
    \end{equation*}
\end{proof}

\subsection{Archimedean closeness of ultrafilters}

If we identify ultrafilters with hypernatural numbers, it is natural to ask when two ultrafilters have generators that are in the same Archimedean class. Equivalently, given $\alpha\in\s\N$, for what $\beta\in \s\N$ there exists $\beta ' \sim \beta$ with $\beta^{\prime}\asymp\alpha$. This question is better expressed in ultrafilter terms, giving rise to the following definition:

\begin{defn}

    Let $\U$ be an ultrafilter on $\N$. We define its Archimedean class $$Arc(\U)=\{\V\in \beta \N \mid \exists m\in \Z\ :\  l_*(\V)=l_*(\U) + m\}$$ where $l=l_2$ and $l_*$ is the pushforward of $l$ to $\beta\N$.
    
\end{defn}

Ultrafilters in the Archimedean class of $\U$ are exactly those generated by $\beta$ as above, where $\alpha\models\U$. 

\begin{lemma}

    Fix $\alpha$ and $\beta$. There exists $\beta'\sim \beta$ in $Arc(\alpha)$ if and only if $\U_\beta\in Arc(\U_\alpha)$.
    
\end{lemma}

\begin{proof}

    If $\beta\sim\beta' \in Arc(\alpha) $ then $l(\beta') = l(\alpha) + m$, $m \in \Z$. Since $\U_{f(\alpha)}=f_*(\U)$ (Theorem \ref{4 Teo su somma di ultrafiltri come somma di ipernaturali}) we have $$l_*(\U_{\beta})=l_*({\U_{\beta'}})=\U_{l(\beta')}=\U_{l(\alpha ) + m}=\U_{l(\alpha)} + m=l_*(\U_\alpha) + m$$

    On the other hand, if $\U_\beta\in Arc(\U_\alpha) $ then $l_*(\U_\alpha)=l_*(\U_\beta) + m $ which means that $l(\alpha)\sim l(\beta) + m$. But now invoke Lemma \ref{2 proprietà generali sim} to find the desired $\beta'$. 
\end{proof}

Using the nonstandard equivalent, we can easily see that 
    \begin{prop}
    Let $n\in \N$ and suppose $l_*(\V)=l_*(\U) + k$. Then there exists $k'\in \Z$ such that $(l_n)_*(\V)=(l_n)_*(\U) + k'$. 
    \end{prop}
So, defining Archimedean classes of ultrafilters via a different base does not change them.\\

We may ask if we can characterize $Arc(\U)$ and what kind of ultrafilters we may find, or not, in it. We start with the second question.

\begin{prop}
    If $n\in\N$ and $\alpha^n\sim \beta \asymp \alpha$, then $n=1$.
\end{prop}

\begin{proof}

    Let apply $l$ to find $$ l(\alpha) + m= l(\beta) \sim l(\alpha^n)=n l(\alpha) + r$$ for some $r\in \Z$.
    So, $l(\alpha) + m$ and $nl(\alpha)$ are in the same class; and since they are equivalent we must have 
    \[
    1= st\left(\frac{ l(\alpha) + m}{n l(\alpha) + r }\right) = \frac{1}{n} 
    \]
    
    so $n=1 $. 
\end{proof}

The same proof also adapts to the following:

\begin{prop}
    If $\xi\in \s\R_{>0}$ if finite and not infinitesimal, and $\alpha^\xi\sim \beta \asymp \alpha$, then $st(\xi)=1$.
\end{prop}


    

\begin{coroll}

    If $f_r:x \to x^r$ for $r\in \R$, then $(f_{r})_{*}(\U)=\U^r\notin Arc(\U)$, unless $r=1$. 
    
    Similarly for $P$ a polynomial of degree different from 1.
    
\end{coroll}

\begin{coroll}

    If $\V\in Arc(\U^r)$ and $r\not=1$, then $\V\notin Arc(\U)$
    
\end{coroll}

\begin{proof}
    In nonstandard terms, let suppose $\beta$ generates $\V$, $\alpha $ generates $\U$ and there exist $\beta', \beta''$ both equivalent to $\beta$ such that 
    \begin{itemize}
        \item $\beta' \asymp \alpha^r$, and
        \item $\beta'' \asymp \alpha$
    \end{itemize}
    So $l(\beta)\sim l(\beta')= l(\alpha^r) + m= rl(\alpha) + m'$ and $l(\beta) \sim l(\beta'')= l(\alpha) + n$: so we find $l(\alpha) + n \sim rl(\alpha) + m'$, which is not possible.
\end{proof}

So we find disjoint classes of ultrafilters looking at their powers: of course, again, taking polynomials or functions that increase polynomially gives rise to the same result. 

\begin{defn}

    Given $A\sub\N$, we define its logarithmic expansion as $$LE(A)=l^{-1}(l(A))=\bigcup_{i\in I} [2^i, 2^{i+1})$$ where $I=\{i \mid \exists a\in A, 2^i\leq a<2^{i+1}\}=l(A)$, and in general for $n\in \Z$ $$LE^{(n)}(A)=\bigcup_{i\in I, i-n\geq 0} [2^{i -n}, 2^{i -n +1})$$    
    
\end{defn}

So $LE(A)$ is obtained ``exploding" every point of $A$ to the whole 2-adic interval it is contained in. For example if $7\in A$, the whole $[4,7]=l^{-1}(l(7))$ will be contained in $LE(A)$; for $LE^{(m)}(A)$ it is enough to shift these intervals: for example if again $7\in A$, $[2,3]=l^{-1}(l(7) - 1)$ will be contained in $LE^{(1)}(A)$.

\begin{rmrk}

    $LE^{(m)}(A)=l^{-1}(l(A) - m)$.
    
\end{rmrk}

Indeed, it is almost by definition that $x\in LE^{(m)}(A) \Leftrightarrow l(x)\in l(A) - m$.

From here we can give an equivalent characterization of $Arc(\U)$. 

\begin{prop}

    $\V\in Arc(\U)$ if and only if $\V \supseteq LE^{(m)}(\U)=\{LE^{(m)}(A) \mid A\in \U\}$ for some $m\in \Z$.
    
\end{prop}

\begin{proof}
    $\Rightarrow)$ By definition, $\V\in Arc(\U)$ means that $l_*(\V)= l_*(\U) + m$ for some $m\in \Z$. Let $A\in \U$. Then $l(A) + m \in l_*(\U) + m =l_*(\V)$ from which we get $l^{-1}(l(A )+ m)=LE^{(-m)}(A)\in \V$.

    $\Leftarrow)$ Suppose $\V\supseteq LE^{(m)}(\U)$: we want to prove that $l_*(\V)=l_*(\U) + m$. For every $A\in \U$, $LE^{(m)}(A)=l^{-1}(l(A) + m)\in \V$ and so $l(l^{-1}(l(A) + m)) \in l_*(\V)$; but $l(l^{-1}(l(A) + m))\sub l(A) +m$, hence $\forall A\in \U\ \ l(A)\in l_*(\V) - m$, and since $\{l(A) \mid A\in \U\}$ generates $l_*(\U)$ we have $l_*(\U)\sub l_*(\V) - m$ and we conclude by maximality.
\end{proof}

\begin{prop}
    $Arc(\U)$ is a left ideal of $(\beta\N, \oplus )$ that is neither open nor closed. 
\end{prop}

\begin{proof}

    Let $\V\in Arc(\U)$ and $\mathcal{W}\in \beta\N$. We want to show that $\mathcal{W}\oplus\V\in Arc(\U)$.

    By hypothesis we know that there exist $\alpha\models \U,\ \beta\models \V$ in the same class. We want to find a generator of $\mathcal{W}\oplus \V$ in the same class of a generator of $\U$. 
    By Proposition \ref{3 Prop ci sono tanti tensori}, take $\gamma\models\mathcal{W} $ such that $(\gamma, \beta)$ is a tensor pair. Then $\gamma + \beta\models \mathcal{W} \oplus \V$ and, by Theorem \ref{4 Puritz}, $\gamma + \beta\asymp \beta$. Hence $\gamma+\beta\asymp\alpha$.

    
    To prove that it is not open nor close, notice that $Arc(\U) = l_*^{-1}(l_*(\U) + \Z)$ and that the $\Z$-orbit $l_*(\U) + \Z$ is not open and not closed (Proposition \ref{3 prop sulla chiusura delle orbite}). Suppose $Arc(\U)$ open. Then, since $l_*$ is open, $l_*(Arc(\U))=l_*(\U) + \Z$ would be open: so $Arc(\U)$ cannot be open, and similarly it cannot be closed.
\end{proof}

\begin{rmrk}
    Since $l_*$ is open we have that $\overline{Arc(\U)}= l^{-1}_*(\overline{\U + \Z})= l^{-1}_*(\beta\Z\oplus \U)$ if $\U$ is nonprincipal (if $\U$ is principal, the class is exactly $\N\sub \beta\N$).
\end{rmrk}

\section{Open questions}

This paper is meant to be a first step towards a through study of asymptotic PR. There are several questions to be solved to perform such a study. The first, natural, one is the following. 

\begin{question}\label{q1}
    Given a PR equation $E(x_1,\dots, x_n)$, can we fully characterize all its possible asymptotics? 
\end{question}

The main difficulty in Question \ref{q1} lies in the fact that, for nonlinear equations, we do not have a complete characterization of their PR. Even assuming the given equation to be PR, the arguments in the proof of Theorem \ref{HLI}, that we used in Section \ref{sec stand asymp} to characterize the linear case, does not adapt to the nonlinear case.

Question \ref{q1} has an interesting sub-question.

\begin{question} \label{q2}
    Does it exist a nonlinear Diophantine equation $P(x_1,\dots, x_n)=0$ which is asymptotically PR with three distinct asymptotic classes, but not with two? More in general, given $k\in\N, k\geq 3$, does it exist a nonlinear Diophantine equation that is asymptotically PR with $k$ asymptotic classes, but not with $k-1$? 
\end{question}

Notice that linear equations that are PR can have only two asymptotic classes, as shown in Theorem \ref{thm char linear asymp}; moreover, in all nonlinear examples we have given, we have a largest class and then ``the remaining terms'', which will be split in some asymptotic classes, but we do not really know how and how many. Question \ref{q2} seems related to a known problem in ultrafilters arithmetic. In fact, as re-shown in Section \ref{eq ultra}, we know how to compare first and last terms in a linear sum of ultrafilters. However, we do not have similar characterization for middle terms. For example, the following problem is open (a similar question can be found in \cite{HMS}, Section 4).

\begin{question}\label{q3}
    Is it true that if $a\U\oplus b \U  \oplus c\U= a\U \oplus b^{\prime}\U\oplus c\U$ then $b=b'?$
\end{question}

The above problem is related to asymptotics for reasons similar to those exploited in Section \ref{eq ultra}: if $(\alpha,\beta,\gamma)$ is a tensor triple\footnote{The 3-dimensional analogue of tensor pairs, see e.g. \cite{luperi2019tensor}.} with $\U=\U_{\alpha}=\U_{\beta}=\U_{\gamma}$, then $a\U\oplus b \U  \oplus c\U=\U_{a\alpha+b\beta+c\gamma}, a\U \oplus b^{\prime}\U\oplus c\U=\U_{a\alpha+b^{\prime}\beta+c\gamma}$. As $\alpha\ll\beta\ll\gamma$, Question \ref{q3} is actually a question about asymptotics in disguise.

Finally, by Proposition \ref{2 prop x^n + y^n = z^n-1 PR sse caso con i coefficenti} the remaining open problems regarding the PR of Fermat-Catalan equations would be solved by answering the following question.

\begin{question}\label{q4}
    Let $n\geq 3$. Is $x^n - y^n = z^{n-1}$ PR?
\end{question}

Due to the scarcity of known solutions to this kind of equations, we think that it is natural to conjecture the answer to Question \ref{q4} to be negative. Moreover, it is known that solutions to such equations can be parameterized, see e.g. (\cite{GRECHUK202369}). Therefore, Question \ref{q4} can be equivalently restated in terms of PR polynomial configurations; unfortunately, the configurations arising from these parametrization are not among those whose PR is known, and do not seem to be much simpler to study than Question \ref{q4} itself.

\end{document}